\documentclass[11pt]{amsart}

\usepackage[thinlines]{easytable}
\usepackage{amsmath,amssymb,amsfonts} \usepackage{float}
\usepackage{graphicx}
\usepackage{booktabs}

 \usepackage{tikz}
 \usetikzlibrary{arrows.meta}

\usepackage{amssymb}
\usepackage{amsthm}
\usepackage{mathrsfs}
\usepackage{comment}
\usepackage{hyperref}

\usepackage[all,cmtip]{xy}\usepackage{xcolor}
\usepackage{enumerate}
\usepackage{bm}
\usepackage{dsfont}
\usepackage{mathtools}
\usepackage[cal=euler]{mathalfa}
\usepackage[top=1in, bottom=1.25in, left=1.25in, right=1.25in]{geometry}
\usepackage{parskip}

\hypersetup{colorlinks=true,linkcolor=magenta,citecolor=blue}

\DeclareMathAlphabet{\mathpzc}{OT1}{pzc}{m}{it}

\theoremstyle{plain}
\newtheorem{theorem}{Theorem}[section]
\newtheorem*{theorem*}{Theorem}

\newtheorem{lemma}[theorem]{Lemma}
\newtheorem*{claim*}{Claim}
\newtheorem{proposition}[theorem]{Proposition}

\newtheorem{corollary}[theorem]{Corollary}

\theoremstyle{definition}

\newtheorem{example}[theorem]{Example}

\newtheorem{remark}[theorem]{Remark}

\numberwithin{equation}{section}
\numberwithin{figure}{section}

\DeclarePairedDelimiter{\norm}{\lVert}{\rVert}
\newcommand{\Pent}{\mathbb{P}}

\newcommand{\Sc}{\operatorname{Sc}}
\newcommand{\Imag}{\operatorname{Im}}
\newcommand{\Span}{\operatorname{span}}
\newcommand{\assoc}[3]{\left[#1,#2,#3\right]}

\newcommand{\R}{\mathbb{R}}
\newcommand{\Q}{\mathbb{Q}}

\newcommand{\C}{\mathbb{C}}

\newcommand{\Hom}{\mathrm{Hom}}

\newcommand{\Tr}{\mathrm{Tr}}

\newcommand{\Aut}{\mathrm{Aut}}

\newcommand{\SO}{\mathrm{SO}}

\begin{document}

\title[Elliptic Integrability and Bracketing-Dependent Dynamics]{A five-dimensional nonassociative algebra: Elliptic Integrability and Bracketing-Dependent Dynamics}

\author{Alireza Poshtkohi}
\address{School of Physics, Engineering and Computer Science, University of Hertfordshire, Hatfield, Hertfordshire, United Kingdom}
\email{a.poshtkohi@herts.ac.uk}

\author{Susanne Pumpl\"un}
\address{School of Mathematical Sciences, University of Nottingham,
Nottingham NG7 2RD, United Kingdom}
\email{Susanne.Pumpluen@nottingham.ac.uk}



\date{\today}

\begin{abstract}
We introduce a five-dimensional nonassociative real algebra \(\Pent\) whose one-generated subalgebras are complex while its left, middle, and right nuclei are all \(\R\). This contrast permits an unambiguous one-element functional calculus but makes polynomial ordinary differential equations sensitive to bracketing. We classify the homogeneous quadratic words containing one fixed coefficient \(B\) and two copies of \(p\) that preserve \(\operatorname{span}_{\R}\{1,B,p\}\). For a fixed-coefficient family of Pentarnion ODEs interpolating between \(B(pp)\) and \((Bp)p\), changing only the bracketing preserves the linearisation but changes the observed global dynamics. Exact balance identities isolate the associator contribution, while numerical diagnostics support bounded chaos at one endpoint and an attracting periodic orbit at the other. The flexibility defect \(p(Bp)-(pB)p=-[p,B,p]\) also produces a rank jump: after polarisation, one-generated subalgebras have dimension three on a nonempty Zariski-open set, rather than at most two. On a canonical invariant three-dimensional space the resulting flow has two first integrals and an explicit periodic orbit on an elliptic curve. We also construct a dense five-dimensional quadratic flow with a full-affine-span Jacobi-elliptic solution. Its projective closure is an elliptic normal sextic, and representation by two ordered products selects the modulus \(m=1/2\).
\end{abstract}

\maketitle

\section{Introduction}

Our main goal is to initiate the investigation of polynomial differential equations over nonassociative locally complex real algebras and to determine how the failure of associativity affects invariant sets, integrability, and global dynamics.

As a test case, we introduce the unital five-dimensional real algebra \(\Pent\). Every subalgebra generated by the unit and one nonreal element is a complex slice, whereas the left, middle, and right nuclei of \(\Pent\) are all \(\R\). The first property makes powers and real-coefficient analytic functions of one element unambiguous. The second means that, once several independent elements interact, the placement of parentheses becomes genuine dynamical data.

Ordinary differential equations over \(\R\) and \(\C\) have been extensively investigated. Over recent decades, special classes of equations have also been studied over quaternion algebras, motivated in part by connections with quantum and fluid mechanics; see, for example, \cite{Adler1986,Adler1995,Finkelsteinetal1962}. Periodic quaternion-valued equations and quaternionic Riccati equations were studied in \cite{CamposMawhin2006,Wilczynski2009}, homogeneous quaternionic polynomial equations in \cite{Gasulletal2009}, and quaternion Bernoulli equations in \cite{Zhang2011}; a recent overview is provided in \cite{WalcherZhang2021}.

Polynomial differential equations over the quaternions are already noncommutative, but the quaternion algebra \(\mathbb H\) remains associative and possesses a nondegenerate multiplicative norm. A relation between homogeneous quadratic differential equations and commutative nonassociative algebras was first established by Markus \cite{Markus1960}. Related approaches appear in \cite{Andreoli1960,Koecher1977,Roehrl1977} and \cite{KaplanYorke1977,Walcher1986}. The monograph \cite{Walcher1991} develops the subject systematically.

Cayley's octonion algebra \(\mathbb O\) is a natural nonassociative choice  and leads to eight-dimensional real systems. It still possesses a nondegenerate multiplicative norm, contains quaternion subalgebras, and while nonassociative, it  is still alternative and flexible. Differential equations over octonions have been considered, for example, in \cite{FrenodLudkovsky2017,Ludkovsky2016,Ludkovsky2010}.

We take a different approach and use a five-dimensional quadratic real algebra that is strongly nonassociative while remaining sufficiently well-behaved for a functional calculus and for algebra-valued ordinary differential equations (ODEs). The algebra \(\Pent\), which we call the \emph{Pentarnion algebra}, is quadratic and has a nondegenerate  positive-definite norm which means it is locally complex. We use \(\Pent\) as a test case for developing polynomial differential equations over quadratic locally complex algebras. Its nuclei are minimal, so interactions involving several independent elements means these will not always associate. 

The classification results in \cite{BresarSemrlSpenko2011} cover locally complex real algebras only through dimension four. To our knowledge, the algebra \(\Pent\) has not previously appeared in the literature.

Quadratic locally complex algebras over \(\R\) form a large family and include quaternion and octonion algebras as special cases. They have a canonical trace, norm, and involution, and every element satisfies a quadratic equation. They are power-associative, which makes powers and one-element polynomial expressions well defined, but they need not be associative and their norm need not be multiplicative. Every real quadratic algebra with a positive-definite norm is locally complex: the subalgebra generated by its unit and any nonreal element is isomorphic to \(\C\). Unlike composition algebras, quadratic locally complex real algebras are not restricted to dimensions \(1,2,4,8\); examples exist in every finite dimension. This dimensional freedom is useful for polynomial differential equations. After fixing the canonical basis of \(\Pent\), each polynomial Pentarnion equation expands uniquely into a structured system of five coupled real polynomial ODEs.

Reference~\cite{RamirezSadovskaia2009} studies invariant elliptic curves in polynomial vector fields. The works \cite{KanekoKuwata2024} and \cite{Shatsila2022} treat projective embeddings by elliptic normal curves. Reference~\cite{SokolovWolf2020} develops matrix formulations of integrable quadratic systems with noncommuting associative variables, while \cite{LevinOlshanetskyZotov2016} develops a Lax formulation. The full-dimensional construction in Equation~\eqref{eq:compact-genus-one-pentarnion}, analysed in Theorem~\ref{thm:invariant-genus-one-ideal}, Corollary~\ref{cor:jacobi-elliptic-orbit}, Proposition~\ref{prop:elliptic-normal-sextic}, and Theorems~\ref{thm:linear-irreducibility} and~\ref{thm:modulus-selection}, combines these themes in a fixed five-dimensional nonalternative algebra. Its invariant oval has full affine span, its projective closure is an elliptic normal sextic, the ambient Jacobian algebra is all of \(M_5(\Q)\), and representability by two ordered Pentarnion products selects the elliptic modulus.

The contents and principal results of the paper are as follows. After briefly recalling basic terminology on nonassociative and locally complex quadratic algebras in Section~\ref{sec:prelim}, we introduce \(\Pent\) in Section~\ref{sec:Pentarnion}, determine its norm and nuclei, prove that its automorphism group is \(SO(2)\) in Theorem~\ref{thm:aut}, and identify canonical complex subalgebras fixed by this action in Corollary~\ref{cor:fixed-subalgebra}. In Section~\ref{sec:calculus} we develop a one-element functional calculus for locally complex quadratic real algebras of arbitrary finite dimension and show that associators of analytic functions of independent arguments reduce to scalar factors multiplying the underlying algebraic associator.

We then turn from associative controls on complex slices to genuinely nonassociative dynamics. In Section~\ref{sec:complex-slice}, Proposition~\ref{prop:complex-slice} proves that a fully bracketed polynomial map whose coefficients belong to \(\Pent^{\Aut(\Pent)}\) is \(SO(2)\)-equivariant and that a canonical complex slice is invariant when it contains the initial value. Proposition~\ref{prop:exact-circle} gives an exact circular family for one bracketing. Theorem~\ref{thm:circle-drift} then computes an exact associator-induced drift and proves that every nonzero deformation parameter excludes periodic orbits in the region considered. Thus changing only the associator weight destroys the displayed recurrent family.

Section~\ref{sec:rank-four-riccati} classifies all orderings and bracketings of the homogeneous word \(Bpp\) that preserve \(\operatorname{span}_{\R}\{1,B,p\}\); see Theorem~\ref{thm:universal-three-plane-classification}. The term
\[
 \mathcal H_B(p)=p(Bp)-(pB)p=-[p,B,p]
\]
measures the failure of flexibility and will be called the \emph{flexibility defect}. In the polarised commutative ODE algebra, \emph{rank three} means that every one-generated subalgebra has dimension at most two. Theorem~\ref{thm:associator-induced-rank-four} shows that adding the flexibility defect creates a third generated direction and hence produces generic rank four: on a nonempty Zariski-open set of generators, the corresponding one-generated subalgebra has dimension three. On the canonical invariant leaf, Theorem~\ref{thm:generic-rank-four-integrability} gives two independent first integrals and a reduction to quadrature, while Theorem~\ref{thm:rank-four-example-periodic-orbit} identifies an explicit bounded periodic orbit on an elliptic curve.

In Section~\ref{sec:pentarnion-odes} we analyse two mechanisms using all five coordinates. First, a dense quadratic field encoded by two ordered products has an exact Jacobi-elliptic solution of full affine span; its projective closure is an elliptic normal sextic, the ambient Jacobian family has no fixed linear block decomposition, and the permitted Pentarnion representation selects the modulus \(m=1/2\). Secondly, for the fixed-coefficient family
\[
 \dot p=Ap+(1-\eta)B(pp)+\eta(Bp)p
       =Ap+B(pp)+\eta[B,p,p],
\]
only the nonlinear associator contribution varies with \(\eta\). The linearisation at the origin is independent of \(\eta\), whereas exact balance identities and numerical Lyapunov, Poincar\'e, Floquet, and continuation calculations distinguish the global dynamics of the two endpoint bracketings. The evidence for chaos is numerical rather than a topological proof. Appendix~\ref{app:associator-235} gives a complementary solvable \(2{:}3{:}5\) example in which associator compensation creates an exact full-span periodic orbit and every uncompensated member changes its initial tangent. Appendix~\ref{app:code-hierarchy} provides the code and data availability information.

\section{Basic terminology and results for nonassociative algebras} \label{sec:prelim}

\subsection{Nonassociative algebras}
Let $F$ be a field and let $A$ be an $F$-vector space. $A$ is called an \emph{algebra} over $F$, if there exists an $F$-bilinear
map $A\times A\mapsto A$, $(x,y) \mapsto x \cdot y$, usually denoted
by juxtaposition, the  \emph{multiplication} of $A$. An algebra
$A$ is  \emph{unital} if there is an element in $A$, denoted by $1$ or $1_A$, such that $1\cdot x=x\cdot 1=x$ for all $x\in A$.

 The algebra $A$ is called \emph{commutative} if $xy=yx$ for all $x,y\in A$, and  \emph{anti-commutative} if $xy=-yx$ for all $x,y\in A$; $A$ is called \emph{flexible} if $x(yx)=(xy)x$ for all  $x,y\in A$, and \emph{alternative} if it is both \emph{left alternative}, i.e. we have  $(xx)y=x(xy)$  and \emph{right alternative}, i.e. we have $y(xx)=(yx)x$, for all  $x,y\in A$. An algebra $A$ is called \emph{power-associative} if $x^k$ is well defined for all integers $k>1$, no matter the placement of the parentheses; equivalently, every subalgebra generated by one element and 1 is associative.

In an $F$-algebra $A$, commutativity is measured by the \emph{commutator} $[x, y] = xy - yx$ and
associativity is measured by the {\it associator} $[x, y, z] = (xy) z - x (yz)$;
 $A$ is called {\it alternative} if its associator $[x, y, z]$ is
alternating. Here, $x,y,z\in A$.

 The \emph{left, middle} and \emph{right nucleus} of
$A$ are defined as ${\rm Nuc}_l(A) = \{ x \in A \, \vert \, [x, A, A]
= 0 \}$, ${\rm Nuc}_m(A) = \{ x \in A \,
\vert \, [A, x, A]  = 0 \}$ and  ${\rm
Nuc}_r(A) = \{ x \in A \, \vert \, [A,A, x]  = 0 \}$. All three sets are associative
subalgebras of $A$. Their intersection
 ${\rm Nuc}(A) = \{ x \in A \, \vert \, [x, A, A] = [A, x, A] = [A,A, x] = 0 \}$ is called the \emph{nucleus} of $A$. Whenever one of the elements
 $x, y, z$ lies in
${\rm Nuc}(A)$, we know that $x(yz) = (xy) z$. 
The  \emph{commuter} of $A$ is defined as ${\rm Comm}(A)=\{x\in A\,|\,xy=yx \text{ for all }y\in A\}$ and the {\it
center} of $A$ is defined as ${\rm C}(A)=\text{Nuc}(A)\cap  {\rm Comm}(A)$.

A nonassociative algebra $A\not=0$ has no zero divisors if and only if
$L_a$ and $R_a$ (left and right multiplication with $a$ in $A$) are injective for all $0\not=a\in A$.   An algebra $A\not=0$  is called a 
\emph{division algebra} if for all $a\in A$, $a\not=0$, both
the left and right multiplication $L_a$ and $R_a$, are bijective for all $0\not=a\in A$. When $A$ has finite dimension (which means $A$ is a finite-dimensional $F$-vector space), $A$ is a division algebra if and only if $A$ has no zero divisors \cite{Schafer1995}.
It is well-known that every real division algebra must have dimension 1, 2, 4, or 8.

 Any (ring) anti-automorphism $\sigma:A\to A$ of period 2 is called an {\it involution} on $A$.
An involution on a unital algebra  is called {\it scalar} if all {\it norms} $\sigma(x)x$ are elements of $F1$.
For every scalar involution $\sigma$, $N:A\to F$ defined via $N_A(x)=\sigma(x)x$ is a quadratic form on $A$, respectively, the {\it trace} $\Tr:A\to A$, $\Tr(x)=\sigma(x)+x$ a linear form on $A$ \cite{McCrimmon1985}.

\subsection{Quadratic algebras}\label{subsec:quadr}
A unital algebra $A$ is called {\it quadratic}, if there exists a quadratic form $N \colon A \to F$ such that  
\begin{equation}\label{equ:1}
x^2 - N(1_A, x)x + N(x)1_A = 0 
\end{equation}
for all $x \in A$,
 where $N(x,y) = N (x+y) - N(x) -N(y)$ denotes the
 symmetric bilinear form induced by $N$. This automatically implies that
$N(1_A) = 1$. The form $N$ is uniquely determined
and called the {\it norm} $N$ of the quadratic algebra $A$.
 The existence of a scalar involution  on a unital algebra $A$ implies that $A$ is quadratic.
 Moreover, the unital algebras with scalar involution are precisely those quadratic algebras with normal trace \cite[Theorem 1.1, Proposition 2.2]{McCrimmon1985}.

 Two isomorphic quadratic algebras over $F$ have isometric norms.
The converse is not true.
 
 Every quadratic algebra is power-associative, so in particular, multiplying Equation \ref{equ:1} on the left or right by $x^n$ yields
\begin{equation}\label{equ:power-recurrence}
x^{n+2} - N(1_A, x)x^{n+1} + N(x)1_A x^n = 0
\end{equation}
for all $x \in A$.

An algebra $A$ together with a nondegenerate quadratic form $N:A\to F$ is called a \emph{mono-composition algebra} if $N(x^2)=N(x)^2$ for all $x\in A$. Every quadratic algebra is a mono-composition algebra \cite[p.~86]{McCrimmon1985}.

When $A$ is a quadratic algebra, then every $x\in A$ with $N(x)\not =0$, has an inverse $x^{-1}$ defined by
\[
x^{-1}=\frac{1}{N(x)} \overline{x},
\]
that means
\[
  xx^{-1}=x^{-1}x=1.
\]

Moreover, then every $x\in A$ generates an associative subalgebra of dimension $\leq 2$ in $A$ which is spanned by $1$ and $x$.

If $A$ is a quadratic algebra over $F$, then $A=F1\oplus V$ with $V=\{ u\in A\,|\, \Tr(u)=0\}$ the pure part of $A$. Define $\times: V\times V\rightarrow V$,
$u\times v=pr(uv)$ with $pr:A \rightarrow V$ the canonical projection map. Then $(V,\times)$ is an anti-commutative algebra over $F$ with $$uv=-\frac{1}{2}N(u,v)+u\times v.$$
The algebra $(V,\times)$ is called the \emph{vector algebra} of $A$ or the  \emph{pure part} of $A$.

Conversely, it is well-known that every quadratic algebra over $F$ 
can be obtained out of an anti-commutative algebra
$(V,\wedge)$ and an $F$-bilinear form $(\,,\,):V\times V\to F$ by defining a multiplication on the $F$-vector
space $A=F\oplus V$ via
$$(a,u)(b,v)=(ab+(u,v),av+bu+u\wedge v) $$
for all $a,b\in F$ and $u,v\in V$. The resulting algebra $A$ is often denoted by $(V,(\,,\,),\wedge)$.

For a quadratic algebra $A$, its symmetrization given by  the Jordan product
\[
x\circ y=\frac{xy+yx}{2}
\]
 makes $(A,\circ)$ into a 
Jordan algebra. 
We have
\[
(a,u)\circ(b,v)
=(ab-\langle u,v\rangle,a v+b u)
\]
for $x=(a,u)$, $y=(b,v)$, so $(A,\circ)$ is the Jordan algebra defined via the symmetric bilinear form $(\,,\,):V\times V\to F$.

\subsection{Composition algebras}  A quadratic form $N \colon A \to F$ is called {\it nondegenerate} if its
 induced symmetric bilinear form $N(u,v)  =
N(u+v)-N(u)-N(v)$ determines an $F$-vector space isomorphism
$C \overset{\sim}\to A^\vee=\Hom_F(A,F)$. A quadratic form $N \colon A \to F$ on a unital algebra $A$ is {\it multiplicative}  if  $N(uv)=N(u)N(v)$ for all $u,v\in A$.
A unital algebra $A$ is called a {\it composition algebra} over $F$
if it admits a nondegenerate multiplicative quadratic form $N \colon A \to F$. Composition algebras are quadratic alternative and $N(1_A)=1$.  A nondegenerate  multiplicative  quadratic form $N$ on  a unital composition algebra $A$ is its norm as a quadratic algebra, thus uniquely determined up to isometry \cite[p.~454 ff.]{KMRT}
 and called the {\it norm} of $A$. 

 A unital quadratic alternative algebra is a composition algebra if
and only if its norm is nondegenerate [M1, 4.6]. Composition algebras
exist only in dimensions 1, 2, 4 or 8. Those of 
dimension 4 are exactly the well-known
 quaternion algebras, those of dimension 8 are called {\it octonion algebras}. The conjugation
 $\overline{x} = \Tr_{C/F}(x)1_A - x$ of a composition algebra
 $A$ is a scalar involution, called the {\it canonical involution} of $A$, where $\Tr_A \colon A \to F$, $\Tr_{A/F}(x) = N_{A/F}
 (1_A, x)$, is the {\it trace} of $A$.

We will mostly work with $F=\mathbb{R}$, but most of our results carry over to any field $F$ of characteristic not 2 in an obvious fashion. When $F=\mathbb{R}$, we call a unital $F$-algebra $A$ \emph{locally complex} if every element $x\in A \setminus F$, together with 1, generates a subalgebra of $A$ (also called a \emph{one-generated} subalgebra) that is isomorphic to $\mathbb{C}$.

\section{The Pentarnion algebra}\label{sec:Pentarnion}

The \emph{Pentarnion algebra} $\Pent=\R\oplus \R i\oplus \R j \oplus \R k\oplus \R \ell$ is a unital 5-dimensional real algebra with basis 1, $i$, $j$, $k$, $\ell$ that is defined via  the following multiplication table.

\begin{table}[H]
\centering
\renewcommand{\arraystretch}{1.35}
$
\begin{array}{c|ccccc}
\cdot & 1 & i & j & k & \ell \\
\hline
1 & 1 & i & j & k & \ell \\
i & i & -1 & \tfrac{k+\ell}{\sqrt{2}} & -\tfrac{j+\ell}{\sqrt{2}} & -\tfrac{j+k}{\sqrt{2}} \\
j & j & -\tfrac{k+\ell}{\sqrt{2}} & -1 & \tfrac{i+\ell}{\sqrt{2}} & -\tfrac{k+i}{\sqrt{2}} \\
k & k & \tfrac{j+\ell}{\sqrt{2}} & -\tfrac{i+\ell}{\sqrt{2}} & -1 & -\tfrac{i+j}{\sqrt{2}} \\
\ell & \ell & \tfrac{j+k}{\sqrt{2}} & \tfrac{k+i}{\sqrt{2}} & \tfrac{i+j}{\sqrt{2}} & -1
\end{array}
$
\caption{Canonical Pentarnion multiplication table.}
\label{tab:multiplication}
\end{table}

Note that as usual, rows multiply columns. Thus for instance the entry in row $i$ and column $j$ is $ij$. The basis 1, $i$, $j$, $k$, $\ell$ is called the \emph{canonical basis} of $\Pent$. Note also while the algebra's multiplication table looks rather balanced, the role of $\ell$ is different from the one of $i,j,k$. This forces the algebra to have a rather small automorphism group, see Theorem \ref{thm:aut}.

This algebra is clearly not commutative (e.g.,
$ ij=\frac{k+\ell}{\sqrt{2}},$ while
 $ji=-\frac{k+\ell}{\sqrt{2}}$), and we will see it is also highly nonassociative, hence we need to carefully manage any bracketing of factors in $\Pent$. 

 \begin{example} We have
$ (ij)k=\frac{1}{\sqrt{2}}(k+\ell)k  =\frac{1}{\sqrt{2}}(-1+\ell k),$
and hence $ (ij)k=-\frac{1}{\sqrt{2}}+\frac{1}{2}i+\frac{1}{2}j.$
On the other hand, $ i(jk)=\frac{1}{\sqrt{2}}i(i+\ell)  =\frac{1}{\sqrt{2}}(-1+i\ell),$
and hence $ i(jk)=-\frac{1}{\sqrt{2}}-\frac{1}{2}j-\frac{1}{2}k$, so that $ (ij)k\not=i(jk)$.
\end{example}

An element  $p\in \Pent$ is called a \emph{Pentarnion}.  Every $p\in \Pent$  can be written as
\[
  p=p_0+p_1i+p_2j+p_3k+p_4\ell,
  \qquad
  p_0,p_1,p_2,p_3,p_4\in\R.
\]
We call the  coefficient $p_0$  the \emph{scalar part}, and
\[
  \Imag(p)=p_1i+p_2j+p_3k+p_4\ell
\]
 the \emph{pure part} of $p$.

 \begin{lemma}
 (i) $\Pent$ has zero divisors.
 \\ (ii) $\Pent$ is not alternative and not flexible.
 \\ (iii) The map
\[
  p=p_0+p_1i+p_2j+p_3k+p_4\ell \mapsto  \overline{p}=p_0-p_1i-p_2j-p_3k-p_4\ell,
\]
is a scalar involution on $\Pent$ and
$$p\overline{p}=p_0^2+p_1^2+p_2^2+p_3^2+p_4^2.$$
\end{lemma}

 \begin{proof}
  $(i)$ This is clear since $\Pent$ has dimension 5.
 \\ $(ii)$  This  follows directly from the general theory of composition and quadratic algebras. For example, $\assoc{i}{i}{j}=(ii)j-i(ij)=\frac{k+\ell}{2}\ne 0$ and $\assoc{i+j}{\ell}{i+j}=2\sqrt{2}$.
\\
 $(iii)$ It is clear that $\overline{\phantom{p}}$ is additive and bijective, and has order 2; i.e. $\overline{\overline{p}}=p$. Moreover, the multiplication table shows that
\[
\overline{pr}=\overline{r}\,\overline{p}
\]
for all $p,r\in \Pent$, so that this map is an anti-automorphism.

For $p=p_0+p_1i+p_2j+p_3k+p_4\ell$, we compute
$$p\overline{p}=p_0^2+p_1^2+p_2^2+p_3^2+p_4^2\in \R,$$
thus $\overline{\phantom{p}}$ is scalar: Write
\[
  p=p_0+u,
  \qquad
  u=p_1i+p_2j+p_3k+p_4\ell\in V.
\]
Then $ \overline{p}=p_0-u$
and, because real scalars commute and associate with every element,
\[
  p\overline{p}=(p_0+u)(p_0-u)=p_0^2-u^2.
\]
Expanding $u^2$ gives  four diagonal terms and six mixed pairs. Put
\[
  D=p_1^2i^2+p_2^2j^2+p_3^2k^2+p_4^2\ell^2
\]
and
\begin{align*}
  M={}&p_1p_2(ij+ji)+p_1p_3(ik+ki)+p_1p_4(i\ell+\ell i)\\
     &+p_2p_3(jk+kj)+p_2p_4(j\ell+\ell j)+p_3p_4(k\ell+\ell k).
\end{align*}
Hence we have
$ u^2=D+M.$
Every mixed pair vanishes because distinct pure basis units anti-commute:
\[
  ij+ji=ik+ki=i\ell+\ell i=jk+kj=j\ell+\ell j=k\ell+\ell k=0.
\]
Since  $i^2=j^2=k^2=\ell^2=-1,$ we get
$u^2=-\left(p_1^2+p_2^2+p_3^2+p_4^2\right)$
  and therefore
\[
N(p)=p\overline{p}=p_0^2+p_1^2+p_2^2+p_3^2+p_4^2.
\]
 \end{proof}
 Since $\Pent$ has a scalar involution, there exists a quadratic form  $N:\Pent\to \mathbb{R}$,
\[
  N(p)=p\overline{p}=p_0^2+p_1^2+p_2^2+p_3^2+p_4^2,
\]
 the  \emph{norm} of $\Pent$, and a linear form  $\Tr:\Pent\to \mathbb{R}$,
\[
  \Tr(p)=p+\overline{p}=2p_0.
\]
 the  \emph{trace}  of $\Pent$ \cite[p.~86]{McCrimmon1985}. We note that our $N$ is nondegenerate, anisotropic, positive definite and that 
\[
  \norm{p}=\sqrt{N(p)},
\] 
where $\norm{\phantom{p}}$ is the Euclidean length of the coordinate vector of $p$. Also note that 
$$N(p)=p\overline{p}=\overline{p}p$$
 and that $N$ is not multiplicative (if it were we would have found a 5-dimensional algebra with a nondegenerate multiplicative norm, a contradiction) but satisfies 
 $$N(p^2)=N(p)^2.$$

\begin{example}
Since $\Pent$ has dimension 5, it has non-trivial zero divisors, e.g.
$(j-k-\ell)(-3\sqrt{2}+i-4k+4\ell)=0$. Note that $N(j-k-\ell)=3$, $N(-3\sqrt{2}+i-4k+4\ell)=51$, hence both $j-k-\ell$ and $-3\sqrt{2}+i-4k+4\ell$ are invertible. However, $p_0^{-1}(p_0r_0)=0$ and $(p_0^{-1}p_0)r_0=r_0$, so that their difference is exactly the associator
\[
  \assoc{p_0^{-1}}{p_0}{r_0}=r_0.
\]
\end{example}

 \begin{remark} $(i)$  The scalar part $\Sc(p)$ and pure part $\Imag(p)$ of $p$ can  be recovered as
\[
  \Sc(p)=\frac{p+\overline{p}}{2}=p_0,
\]
\[
  \Imag(p)=\frac{p-\overline{p}}{2}=p_1i+p_2j+p_3k+p_4\ell.
\]
$(ii)$ The
product of two distinct pure
basis elements has Euclidean length one; for example,
$ N(ij)=N\!\left(\frac{k+\ell}{\sqrt{2}}\right)=1.$
 \\ $(iii)$
  The sign convention of the products of the basis vectors can be encoded by the following directed graph on the four pure canonical basis elements.

\begin{figure}[H]
\centering
\begin{tikzpicture}

  \node[circle, draw, ultra thick, minimum size=14mm] (k) at (0,3.3) {\Large $k$};
  \node[circle, draw, ultra thick, minimum size=14mm] (i) at (-2.7,0) {\Large $i$};
  \node[circle, draw, ultra thick, minimum size=14mm] (j) at (2.7,0) {\Large $j$};
  \node[circle, draw, ultra thick, minimum size=14mm] (ell) at (0,1.25) {\Large $\ell$};

  \path[->, >=stealth, ultra thick] (i) edge[bend left=5] (j);
  \path[->, >=stealth, ultra thick] (j) edge[bend left=5] (k);
  \path[->, >=stealth, ultra thick] (k) edge[bend left=5] (i);

  \path[->, >=stealth, thick] (ell) edge (i);
  \path[->, >=stealth, thick] (ell) edge (j);
  \path[->, >=stealth, thick] (ell) edge (k);
\end{tikzpicture}
\caption{Directed-graph encoding of the edges $i\to j\to k\to i$ and $\ell\to i$, $\ell\to j$, $\ell\to k$.}
\label{fig:multiplication-graph}
\end{figure}
For each directed edge $a\to b$, let $c$ and $d$ be the two remaining pure basis units and set
$ab=\frac{c+d}{\sqrt{2}}.$ Reversing the arrow changes the sign. Together with $i^2=j^2=k^2=\ell^2=-1$, the identity element, anti-commutativity, and bilinearity, this graph displays the full multiplication table of $i,j,k,l$.
\end{remark}

 \begin{corollary}\label{cor:quadratic}
(i) $\Pent$ is a unital quadratic algebra; every $p\in \Pent$ satisfies the quadratic equation
$$ p^2-\Tr(p)p+N(p)1=0.$$
 (ii) $\Pent$ is power associative.
\\ (iii) $p^{n+2}=\Tr(p)p^{n+1}-N(p)p^n=2p_0p^{n+1}-N(p)p^n$ for all $ n\geq 0.$
\\ (iv) $\Pent$ is a locally complex algebra.
\end{corollary}
 
 \begin{proof}
$(i)$  This follows from $(ii)$ by \cite{McCrimmon1985}, since $\Pent$ carries a scalar involution.
\\ $(ii)$
Since $\Pent $ is a unital quadratic algebra, $\Pent$ is power associative.
\\ $(iii)$ Since $ p^2=\Tr(p)p-N(p)1$, this is straightforward.
\\ $(iv)$ Since $\Pent$ is a quadratic algebra, this is well known, but we give a direct proof here as well: Every subalgebra generated by $1$ and some pure element will be a (unital) associative algebra of dimension $2$. Thus it is well known that every subalgebra will either be $\mathbb{C}$, $\R[t]/(t^2)$ or the split algebra $\R\times \R$. Let $\R(v)$ be such a subalgebra with $v\not=0$, and note that  $v^2=-N(v)$. Since $N(v)$ is an anisotropic form, we have $N(v)=0$ if and only if $v=0$, so that $N(v)\not=0$ for all $v\in V$, which excludes the case that $\R(v)$ are the dual numbers. Also $N$ is positive definite, so that the split algebra $\R\times \R$ will not appear either as always $N(v)>0$. This yields the assertion.
\end{proof}

Since $\Pent$ is locally complex, every pure Pentarnion $p\in V$ determines a complex-like slice. This means that real-coefficient analytic functions of one Pentarnion $p$ are unambiguous.

Since $\Pent$ is a quadratic algebra, by Section \ref{subsec:quadr} we have
 $\Pent=\R 1\oplus V$ with the pure part $V=\{ u\in \Pent\,|\, \Tr(u)=0\}=\Span_{\R}\{i,j,k,\ell\}$. Define $\times: V\times V\rightarrow V$,
$u\times v=pr(uv)$ with $pr:\Pent \rightarrow V$ the canonical projection map. Then the pure part $(V,\times)$ of $\Pent$ is an anti-commutative algebra over $F$ of dimension 4 and
$$uv=-\frac{1}{2}N(u,v)+u\times v$$
for all $u,v\in V$.

More concretely,  $ V=\Span_{\R}\{i,j,k,\ell\}$ is equipped
with the Euclidean inner product $\langle\,,\,\rangle$ for which $i,j,k,\ell$ are orthonormal. For pure elements $u,v\in V$, define the anti-symmetric map
\[
  \Pi:\Lambda^2V\longrightarrow V
\]
by its values on the six basis bi-vectors
\[
  \Pi(i\wedge j)=\frac{k+\ell}{\sqrt{2}},
  \qquad
  \Pi(j\wedge k)=\frac{i+\ell}{\sqrt{2}},
  \qquad
  \Pi(k\wedge i)=\frac{j+\ell}{\sqrt{2}},
\]
\[
  \Pi(\ell\wedge i)=\frac{j+k}{\sqrt{2}},
  \qquad
  \Pi(\ell\wedge j)=\frac{k+i}{\sqrt{2}},
  \qquad
  \Pi(\ell\wedge k)=\frac{i+j}{\sqrt{2}}.
\]
Then we can also write
$$u\times v=\Pi(u\wedge v)$$
and the  multiplication of $\Pent$ restricted to  the pure part $V$ can then be written as
\[
uv=-\langle u,v\rangle+\Pi(u\wedge v)=-\frac{1}{2}N(u,v)+\Pi(u\wedge v).
\]
The scalar part is the negative of the symmetric bilinear form  $\frac{1}{2}N(u,v)$ associated to $N$, the pure part is anti-symmetric. Note that
\[
  uv+vu=-2\langle u,v\rangle,
\]
and
\[
  uv-vu=2\Pi(u\wedge v).
\]
Setting $v=u$ recovers $u^2=-\norm{u}^2$, because $u\wedge u=0$.

 By Section \ref{subsec:quadr}, the multiplication of $\Pent=\R\oplus V$ can subsequently be written as
 $$(a,u)(b,v)=(ab-\langle u,v\rangle,av+bu+\Pi(u\wedge v)) $$
 for all $a,b\in \R$ and $u,v\in V$ with the $\R$-bilinear form $\langle,\rangle:V\times V\to \R$, $\langle u, v\rangle=\frac{1}{2}N(u,v)$ which is associated to $N$, that is the Euclidean product.
This formula  is  the starting point for converting Pentarnion equations into equations with real coordinates.

\begin{remark} (Geometric meaning of complementary-plane multiplication)
Since $ \dim V=4,$ the space of bi-vectors has dimension
$ \dim\Lambda^2V=\binom{4}{2}=6$ and the map
\[
  \Pi:\Lambda^2V\longrightarrow V
\]
therefore sends six oriented two-dimensional plane directions into four pure directions.
For example,
\[
  i\wedge j\longmapsto
  \Pi(i\wedge j)=\frac{k+\ell}{\sqrt{2}}
  \in\Span\{k,\ell\}.
\]
This means 
\[
  e_a e_b\in
  \Span\{\text{the complementary pair to }e_a,e_b\}
\]
if $e_a$ and $e_b$ denote two canonical basis vectors in $V$.
\end{remark}

Define the symmetrised product
\[
p\circ r=\frac{pr+rp}{2}
\]
on $\Pent$.
The alternating term cancels, giving
\[
(a,u)\circ(b,v)
=(ab-\langle u,v\rangle,a v+b u).
\]
This is the five-dimensional locally complex spin-factor-type Jordan product, and $(\Pent, \circ)$ is often denoted $J_{5}$ in the locally complex algebra literature. Thus $\Pent$ is \emph{Jordan-admissible}: its symmetrisation is a Jordan algebra. Over fields of characteristic different from $2$, noncommutative Jordan algebras are flexible Jordan-admissible algebras. The Pentarnion algebra $\Pent$ is not flexible.
Therefore $\Pent$ is not a (noncommutative) Jordan algebra. We summarize these observations.

\begin{lemma}\label{le:Jordanadm}
$\Pent$ is Jordan-admissible  and $(\Pent,\circ)=J_5$.
\end{lemma}

\begin{remark}
A direct implication, relevant for working with ODEs, is that $(\Pent,\circ)=J_5$ is a rank three commutative algebra, and these algebras are very well understood  \cite{Walcher1999}.
 Because its anti-symmetric component vanishes during self-multiplication (\(p^2 = p \circ p\)), the equation \(\dot{p}=p^2\) behaves classically and possesses a simple solution. This approach allows us to understand the non-linear dynamics in our later section through the lens of the Jordan algebra $J_5$ (Section \ref{sec:complex-slice}).
\end{remark}

The algebra $\Pent$ is highly nonassociative.

\begin{proposition}
${\rm Nuc}_l(\Pent)={\rm Nuc}_m(\Pent)={\rm Nuc}_r(\Pent)=\R$.
\end{proposition}

\begin{proof}
Since $\R\subset {\rm Nuc}_l(\Pent)$, $\R\subset {\rm Nuc}_m(\Pent)$ and $\R\subset {\rm Nuc}_r(\Pent)$, it remains to show that
these nuclei are contained in $\R$ which is a tedious calculation showing that the pure basis elements are all not contained in the nuclei.
\end{proof}

We now use the basis of $\Pent$ given by
\begin{equation}
\label{eq:adapted-basis}
 \mathbf a=\frac{i+j+k}{\sqrt3},\qquad
 \mathbf b=\frac{i-j}{\sqrt2},\qquad
 \mathbf c=\frac{i+j-2k}{\sqrt6},\qquad
 \mathbf d=\ell.
\end{equation} 

\begin{lemma}
\label{lem:adapted-products}
We have
\[
 \mathbf a^2=\mathbf b^2=\mathbf c^2=\mathbf d^2=-1
\]
and
\begin{align*}
 \mathbf a\mathbf b&=\frac{\mathbf c}{\sqrt2},&
 \mathbf a\mathbf c&=-\frac{\mathbf b}{\sqrt2},&
 \mathbf a\mathbf d&=-\sqrt2\,\mathbf a,\\
 \mathbf b\mathbf c&=\frac{\mathbf a}{\sqrt2}+\sqrt{\frac32}\,\mathbf d,&
 \mathbf b\mathbf d&=\frac{\mathbf b}{\sqrt2},&
 \mathbf c\mathbf d&=\frac{\mathbf c}{\sqrt2}.
\end{align*}
Anti-commutativity yields the reverse products.
\end{lemma}

\begin{proof}
Substitute \eqref{eq:adapted-basis} into the canonical multiplication table
and use bilinearity.  For example,
\[
 \mathbf a\mathbf b
 =\frac1{\sqrt6}(i+j+k)(i-j)=\frac{\mathbf c}{\sqrt2},
\]
\[
 \mathbf a\mathbf d
 =\frac1{\sqrt3}(i\ell+j\ell+k\ell)=-\sqrt2\,\mathbf a,
\]
and
\[
 \mathbf b\mathbf c
 =\frac1{\sqrt{12}}(i-j)(i+j-2k)
 =\frac{\mathbf a}{\sqrt2}+\sqrt{\frac32}\,\mathbf d.
\]
The remaining identities follow similarly.  The change of pure basis is
orthogonal, so the four squares remain \(-1\).
\end{proof}

For \(a\in\R\), define the algebra isomorphism $S_a:\Pent\to \Pent $,
\begin{align}
 S_a(1)&=1,&S_a(\mathbf a)&=\mathbf a,&S_a(\mathbf d)&=\mathbf d,
 \label{eq:R-fixed}\\
 S_a(\mathbf b)&=\cos(a)\,\mathbf b+\sin(a)\,\mathbf c,&
 S_a(\mathbf c)&=-\sin(a)\,\mathbf b+\cos(a)\,\mathbf c.
 \label{eq:R-plane}
\end{align}
We now show that all isomorphisms of $\Pent$ have this form.

\begin{theorem}\label{thm:aut}
${\rm Aut}_F(\Pent)=\{ S_a\,|\, a\in \R \}\cong \SO(2)$.
\end{theorem}

\begin{proof}
Let $\Phi \in{\rm Aut}_F(\Pent)$. Since $\Phi((1,0))=(1,0)$, we know  $\Phi((a,0))=(a,0)$ for all $a\in \R$, so $\Phi|_\R=id$. Now
$$\Phi((a,u)(b,v))=\Phi(ab-\langle u,v\rangle,av+bu+\Pi(u\wedge v))=\Phi((a,u))\Phi((b,v)) $$
is equivalent to 
$$(\Phi(a)\Phi(b)-\Phi((\langle u,v\rangle),a\Phi(v)+b \Phi(u)+\Phi(\Pi(u\wedge v)))=(a,\Phi(u))(b,\Phi(v))$$ 
$$= (ab-\langle \Phi(u),\Phi(v)\rangle,a\Phi(v)+b\Phi(u)+\Pi(\Phi(u)\wedge \Phi(v)))$$
which in turn implies
$$\langle u,v\rangle=\langle \Phi(u),\Phi(v)\rangle  $$
and
\begin{equation}
\label{eq:cross-equivariance}
\Phi(\Pi(u\wedge v))=\Pi(\Phi(u)\wedge \Phi(v)).
\end{equation}
Hence $\Phi|_V $ is an isometry of $N|_V $, since we know that scalar part is the negative of the symmetric bilinear form  $\frac{1}{2}N(u,v)$ associated to $N$. Thus $\Phi$  preserves the inner product and therefore lies in \(O(V)\).

For ease of notation, we write $\Pi(u\wedge v)$ subsequently as $u\times v$. 

We have
\[
 \mathbf a^2=\mathbf b^2=\mathbf c^2=\mathbf d^2=-1
\]
and
\begin{align}
 \mathbf a\mathbf b&=\frac{\mathbf c}{\sqrt2},&
 \mathbf a\mathbf c&=-\frac{\mathbf b}{\sqrt2},&
 \mathbf a\mathbf d&=-\sqrt2\,\mathbf a,\label{eq:adapted1}\\
 \mathbf b\mathbf c&=\frac{\mathbf a}{\sqrt2}+\sqrt{\frac32}\,\mathbf d,&
 \mathbf b\mathbf d&=\frac{\mathbf b}{\sqrt2},&
 \mathbf c\mathbf d&=\frac{\mathbf c}{\sqrt2}.\label{eq:adapted2}
\end{align}
Take the left multiplication
\[
 L_u:V\to V,\qquad L_u(v)=u\times v.
\]
If
\(u=\alpha\mathbf a+\beta\mathbf b+\gamma\mathbf c+\delta\mathbf d\),
then, with respect to
\((\mathbf a,\mathbf b,\mathbf c,\mathbf d)\), $L_u$ is represented by the matrix
\begin{equation}
\label{eq:Lu-matrix}
 \begin{pmatrix}
 \sqrt2\delta&-\gamma/\sqrt2&\beta/\sqrt2&-\sqrt2\alpha\\
 \gamma/\sqrt2&-\delta/\sqrt2&-\alpha/\sqrt2&\beta/\sqrt2\\
 -\beta/\sqrt2&\alpha/\sqrt2&-\delta/\sqrt2&\gamma/\sqrt2\\
 0&-\sqrt{3/2}\,\gamma&\sqrt{3/2}\,\beta&0
 \end{pmatrix}.
\end{equation}
The map \(L_u\) is self-adjoint with respect to the Euclidean inner product
if and only if \(u\in\R\mathbf d\). Let $\operatorname{spec}(L_u)$ denote the spectrum (i.e., the set of eigenvalues) of $L_u$. Then
\[
 \operatorname{spec}(L_{\mathbf d})
 =\left\{\sqrt2,-\frac1{\sqrt2},-\frac1{\sqrt2},0\right\}:
\]
 Equating \eqref{eq:Lu-matrix} with its transpose first gives
\(\gamma=0\) from the \((1,2)\) entry, \(\beta=0\) from the
\((1,3)\) entry, and \(\alpha=0\) from the \((1,4)\) entry.  Conversely,
if \(u=\delta\mathbf d\), the matrix is diagonal and hence self-adjoint.
For \(u=\mathbf d\), \eqref{eq:Lu-matrix} becomes the diagonal matrix
\[
 L_{\mathbf d}
 =\operatorname{diag}\left(\sqrt2,-\frac1{\sqrt2},
 -\frac1{\sqrt2},0\right),
\]
which gives the set of eigenvalues
\[
 \operatorname{spec}(L_{\mathbf d})
 =\left\{\sqrt2,-\frac1{\sqrt2},-\frac1{\sqrt2},0\right\}.
\]
The multiplication in $\Pent$ shows directly that
every \(S_a\) preserves all basis products.  Hence
\(S_a\in\Aut(\Pent)\), and
\(S_a S_b=S_{a+b}\).

Conversely, let \(\Phi\in\Aut(\Pent)\),  then
Equation \eqref{eq:cross-equivariance} implies
\begin{equation}
\label{eq:L-conjugacy}
 L_{\Phi(u)}=\Phi L_u \Phi^{-1}
\end{equation}
for all $u\in V$.
Since orthogonal conjugacy preserves self-adjointness,

\(\Phi(\mathbf d)=\pm\mathbf d\).
The spectra of \(L_{\mathbf d}\) and \(L_{-\mathbf d}=-L_{\mathbf d}\)
are different, while Equation \ref{eq:L-conjugacy} preserves spectra.  Therefore
\(\Phi(\mathbf d)=\mathbf d\).

It follows from \eqref{eq:L-conjugacy} that \(\Phi|_V\) commutes with
\(L_{\mathbf d}\).  The eigenspaces of \(L_{\mathbf d}\) are $\R\mathbf a,$  $E=\Span\{\mathbf b,\mathbf c\},$ and $ \R\mathbf d,$
with respective eigenvalues $\sqrt2,$ $-1/\sqrt2$ and 0.  Hence
\(\Phi(\mathbf a)=\varepsilon\mathbf a\), where \(\varepsilon=\pm1\), and
\(\Phi|_E\in O(2)\).

Now consider
\[
 \mathbf b\times\mathbf c
 =\frac{\mathbf a}{\sqrt2}+\sqrt{\frac32}\,\mathbf d.
\]
Because \(\Lambda^2E\) is one-dimensional,
\[
 \Phi(\mathbf b)\times\Phi(\mathbf c)
 =\det(\Phi|_E)\left(\frac{\mathbf a}{\sqrt2}
               +\sqrt{\frac32}\,\mathbf d\right).
\]
On the other hand, we have
\[
 \Phi(\mathbf b\times\mathbf c)
 =\frac{\varepsilon\mathbf a}{\sqrt2}
  +\sqrt{\frac32}\,\mathbf d.
\]
Comparison of the \(\mathbf d\)-components gives \(\det(\Phi|_E)=1\), and
comparison of the \(\mathbf a\)-components gives \(\varepsilon=1\).
Thus \(\Phi|_E\in SO(2)\), \(\Phi(\mathbf a)=\mathbf a\), and
\(\Phi(\mathbf d)=\mathbf d\).  Consequently \(\Phi=S_a\) for some \(a\in \R\).
\end{proof}

 \begin{corollary}[Fixed subalgebras]
\label{cor:fixed-subalgebra}
The fixed-point algebra of the automorphism group is given by
\[
 \Pent^{\Aut(\Pent)}
 =\Span_{\R}\{1,\mathbf a,\mathbf d\}.
\]
Every subalgebra generated by 1 and a non-real $p\in \Pent^{\Aut(\Pent)}$
is isomorphic to the complex numbers.
\end{corollary}

\begin{proof}
A vector is fixed by every planar rotation if and only if its
\(\mathbf b\)- and \(\mathbf c\)-coordinates vanish.  Closure under
multiplication follows from
\(\mathbf a^2=\mathbf d^2=-1\),
\(\mathbf a\mathbf d=-\sqrt2\mathbf a\), and
\(\mathbf d\mathbf a=\sqrt2\mathbf a\).  The rest of the assertion is trivial.
\end{proof}

In particular, $ \Pent^{\Aut(\Pent)}$ contains the subalgebras
\[
 \mathbb C_{\mathbf a}=\Span_{\R}\{1,\mathbf a\}\cong \C,
 \qquad
 \mathbb C_{\mathbf d}=\Span_{\R}\{1,\mathbf d\}\cong \C.
\]
These isomorphisms are canonical: $\mathbb C_{\mathbf a}$ and
$\mathbb C_{\mathbf d}$ have the ordinary complex multiplication tables.

The following formula is useful for symbolic computation and for translating an equation over $\Pent$ into five coupled real equations.

\begin{lemma}
Let $p=w+xi+yj+zk+t\ell$
and $ r=a+bi+cj+dk+e\ell.$
Then
\[
  pr=A+Bi+Cj+Dk+E\ell,
\]
where
\begingroup
\small
\begin{align*}
  A&=wa-xb-yc-zd-te,\\
  B&=wb+ax+\frac{yd-ye-zc-ze+tc+td}{\sqrt{2}},\\
  C&=wc+ay+\frac{-xd-xe+zb-ze+tb+td}{\sqrt{2}},\\
  D&=wd+az+\frac{xc-xe-yb-ye+tb+tc}{\sqrt{2}},\\
  E&=we+at+\frac{xc-xd-yb+yd+zb-zc}{\sqrt{2}}.
\end{align*}
\endgroup
\end{lemma}

The proof is elementary but tedious.

\section{A  functional calculus employing quadratic locally complex algebras}\label{sec:calculus}

This section gives a first taste for an analytic
extension of a quadratic real algebra $A$ of dimension $n$ with basis $1,e_1,\dots,e_{n-1}$, with scalar canonical involution
\[
  p=p_0+p_1e_1+\dots p_{n-1}e_{n-1}\mapsto  \overline{p}=p_0-p_1e_1-p_2e_2-\dots -p_{n-1}e_{n-1},
\]
and an anisotropic positive definite norm 
$$N(p)=p_0^2+p_1^2+\dots+p_{n-1}^2.$$
 It is deliberately a one-element functional calculus, not yet a general multivariable theory of some analysis using these algebras.
 The key fact we use is that such an algebra is locally complex by definition, so every element $p$ lies in an associative  subalgebra isomorphic to the complex numbers.

We note that our $N$ is nondegenerate, anisotropic, positive definite and that
\[
  \norm{p}=\sqrt{N(p)},
\]
where $\norm{\phantom{p}}$ is the Euclidean length of the coordinate vector $(p_0,\dots,p_{n-1})$ of $p$. We will use the notation from Section \ref{subsec:quadr}.

\subsection{Slice decomposition}

Let $p=a+u,$ $a\in\R,$ $u\in V$.
Set $\rho=\norm{u}$ then $\rho>0$. Define $ I=\frac{u}{\rho},$ then $ I^2=-1$
and
\[
  p=a+\rho I.
\]
All powers of $p$ belong to the associative slice $\Span_{\R}\{1,I\}\cong \mathbb{C}$. Therefore we can write down a well-defined real-coefficient power series
\[
  f(p)=\sum_{n=0}^{\infty}c_np^n
\]
of elements $p\in \Pent$ ($c_n\in \R$) and talk about convergence of the series, employing the Euclidean norm.

\subsection{The exponential series}

Define
\[
  \exp(p)=\sum_{n=0}^{\infty}\frac{p^n}{n!}.
\]
Because every $a\in \R$ commutes with $\rho I$ inside the  slice $\Span_{\R}\{1,I\}$, we have
\[
  \exp(a+\rho I)=e^a\exp(\rho I).
\]
We have
\[
  (\rho I)^{2n}=(-1)^n\rho^{2n},
\]
\[
  (\rho I)^{2n+1}=(-1)^n\rho^{2n+1}I.
\]
Hence we obtain
\[
  \exp(\rho I)
  =\sum_{n=0}^{\infty}\frac{(-1)^n\rho^{2n}}{(2n)!}
   +I\sum_{n=0}^{\infty}\frac{(-1)^n\rho^{2n+1}}{(2n+1)!}.
\]
Recognising the cosine and sine series gives
\[
  \exp(\rho I)=\cos\rho+I\sin\rho,
\]
therefore we get
\[
  \exp(p)=e^a\left(\cos\rho+I\sin\rho\right)
  =e^a\left(\cos\rho+\frac{u}{\rho}\sin\rho\right).
\]
At $\rho=0$, the expression is interpreted by continuity and reduces to $\exp(p)=e^a$.

\subsection{Sine and cosine}

Define
\[
  \sin(p)=\sum_{n=0}^{\infty}\frac{(-1)^np^{2n+1}}{(2n+1)!},
\]
\[
  \cos(p)=\sum_{n=0}^{\infty}\frac{(-1)^np^{2n}}{(2n)!}.
\]
Within the  slice $\Span_{\R}\{1,I\}$, the usual addition formulas are valid. First, the pure argument $\rho I$ satisfies
\[
  \cos(\rho I)
  =\sum_{n=0}^{\infty}\frac{(-1)^n(\rho I)^{2n}}{(2n)!}
  =\sum_{n=0}^{\infty}\frac{\rho^{2n}}{(2n)!}
  =\cosh\rho,
\]
and
\[
  \sin(\rho I)
  =\sum_{n=0}^{\infty}\frac{(-1)^n(\rho I)^{2n+1}}{(2n+1)!}
  =I\sum_{n=0}^{\infty}\frac{\rho^{2n+1}}{(2n+1)!}
  =I\sinh\rho.
\]
Consequently,
\begin{align*}
  \sin(p)
  &=\sin(a+\rho I)\\
  &=\sin a\,\cos(\rho I)+\cos a\,\sin(\rho I)\\
  &=\sin a\,\cosh\rho+I\cos a\,\sinh\rho,
\end{align*}
so
\[
\sin(p)=\sin a\,\cosh\rho
  +\frac{u}{\rho}\cos a\,\sinh\rho.
\]
Similarly,
\begin{align*}
  \cos(p)
  &=\cos(a+\rho I)\\
  &=\cos a\,\cos(\rho I)-\sin a\,\sin(\rho I)\\
  &=\cos a\,\cosh\rho-I\sin a\,\sinh\rho,
\end{align*}
and therefore
\[
\cos(p)=\cos a\,\cosh\rho
  -\frac{u}{\rho}\sin a\,\sinh\rho.
\]
The factors $\sinh\rho/\rho$ are interpreted by their continuous limit $1$ when $\rho=0$.

\subsection{Hyperbolic sine and cosine}

Define
\[
  \sinh(p)=\sum_{n=0}^{\infty}\frac{p^{2n+1}}{(2n+1)!},
\]
\[
  \cosh(p)=\sum_{n=0}^{\infty}\frac{p^{2n}}{(2n)!}.
\]
For the pure slice argument,
\[
  \cosh(\rho I)
  =\sum_{n=0}^{\infty}\frac{(\rho I)^{2n}}{(2n)!}
  =\sum_{n=0}^{\infty}\frac{(-1)^n\rho^{2n}}{(2n)!}
  =\cos\rho,
\]
and
\[
  \sinh(\rho I)
  =\sum_{n=0}^{\infty}\frac{(\rho I)^{2n+1}}{(2n+1)!}
  =I\sum_{n=0}^{\infty}\frac{(-1)^n\rho^{2n+1}}{(2n+1)!}
  =I\sin\rho.
\]
Using the addition formulas in the associative slice gives
\begin{align*}
  \sinh(p)
  &=\sinh(a+\rho I)\\
  &=\sinh a\,\cosh(\rho I)+\cosh a\,\sinh(\rho I)\\
  &=\sinh a\,\cos\rho+I\cosh a\,\sin\rho,
\end{align*}
hence
\[
\sinh(p)=\sinh a\,\cos\rho
  +\frac{u}{\rho}\cosh a\,\sin\rho.
\]
Likewise,
\begin{align*}
  \cosh(p)
  &=\cosh(a+\rho I)\\
  &=\cosh a\,\cosh(\rho I)+\sinh a\,\sinh(\rho I)\\
  &=\cosh a\,\cos\rho+I\sinh a\,\sin\rho,
\end{align*}
so
\[
\cosh(p)=\cosh a\,\cos\rho
  +\frac{u}{\rho}\sinh a\,\sin\rho.
\]

\subsection{Useful special cases and identities}

For  $u\in V$ with $\rho=\norm{u}$, we have
\[
  \exp(u)=\cos\rho+\frac{u}{\rho}\sin\rho,
\]
\[
  \sin(u)=\frac{u}{\rho}\sinh\rho,
  \qquad
  \cos(u)=\cosh\rho,
\]
\[
  \sinh(u)=\frac{u}{\rho}\sin\rho,
  \qquad
  \cosh(u)=\cos\rho.
\]
All familiar identities involving only  functions of the same $p\in A$ remain valid because the functions lie in the same associative slice. In particular,
\[
  \exp(p)\exp(-p)=1,
\]
\[
  \sin^2(p)+\cos^2(p)=1,
\]
and
\[
  \cosh^2(p)-\sinh^2(p)=1.
\]
Similarly,
\[
  \exp(-p)\bigl(\sin(p)\cos(p)\bigr)
  =\bigl(\exp(-p)\sin(p)\bigr)\cos(p),
\]
because all three factors are functions of the same element $p\in A$.
The situation becomes very intricate and interesting when analytic functions are evaluated at independent different elements of $A$ which do not lie in the same slice.   In that case, the
non-associativity of $A $ becomes visible immediately and is highly relevant. We will demonstrate this on the five-dimensional quadratic algebra $\Pent$.

\subsection{Pentarnion functional calculus, analytic functions with independent arguments}
 Let
\[
  p_m=a_m+u_m,
  \qquad
  \rho_m=\norm{u_m},
  \qquad
  m=1,2,3.
\]
Our previous formulas show that
\[
  \exp(p_1)=E_0+E_1u_1,
  \qquad
  E_1=e^{a_1}\frac{\sin\rho_1}{\rho_1},
\]
\[
  \sin(p_2)=S_0+S_1u_2,
  \qquad
  S_1=\cos a_2\frac{\sinh\rho_2}{\rho_2},
\]
and
\[
  \cos(p_3)=C_0+C_1u_3,
  \qquad
  C_1=-\sin a_3\frac{\sinh\rho_3}{\rho_3}.
\]
The associator  vanishes whenever any one of its arguments is a real scalar. Therefore
\[
  \assoc{\exp(p_1)}{\sin(p_2)}{\cos(p_3)}
  =E_1S_1C_1\assoc{u_1}{u_2}{u_3}.
\]
Since scalar parts do not contribute to an associator, we get
\[
  \assoc{u_1}{u_2}{u_3}=\assoc{p_1}{p_2}{p_3}.
\]
Define the real coefficient
\[
  K=-e^{a_1}\cos a_2\sin a_3
  \frac{\sin\rho_1}{\rho_1}
  \frac{\sinh\rho_2}{\rho_2}
  \frac{\sinh\rho_3}{\rho_3}.
\]
Then the exact \emph{analytic associator} is
\[
  \assoc{\exp(p_1)}{\sin(p_2)}{\cos(p_3)}
  =K\assoc{p_1}{p_2}{p_3}.
\]
The quotient functions are as usual interpreted by continuity at zero. This identity provides a direct bridge between one-element analytic functions and the non-associative structure of the full algebra.

\subsection{Pentarnion functional calculus, a non-associative analytic example}

Choose
\[
  p_1=\frac{\pi}{2}i,
  \qquad
  p_2=j,
  \qquad
  p_3=1+k.
\]
The one-element formulas give
\[
  \exp(p_1)=i,
\]
\[
  \sin(p_2)=j\sinh (1),
\]
and
\[
  \cos(p_3)=\cos (1)\,\cosh (1)-k\sin (1)\,\sinh (1).
\]
The scalar part of $\cos(p_3)$ contributes nothing to the associator. Hence
\[
  \assoc{\exp(p_1)}{\sin(p_2)}{\cos(p_3)}
  =-\sin (1)\,\sinh^2 (1)\,\assoc{i}{j}{k}.
\]
Using
\[
  \assoc{i}{j}{k}=\frac{1}{2}i+j+\frac{1}{2}k,
\]
we obtain the exact nonzero result
\[
  \assoc{\exp(p_1)}{\sin(p_2)}{\cos(p_3)}
  =-\sin 1\,\sinh^2 1
   \left(\frac{1}{2}i+j+\frac{1}{2}k\right).
\]
Equivalently,
\[
  \bigl(\exp(p_1)\sin(p_2)\bigr)\cos(p_3)
  \ne
  \exp(p_1)\bigl(\sin(p_2)\cos(p_3)\bigr).
\]
For completeness, set
\[
  S=\sin (1)\,\sinh^2 (1),
  \qquad
  T=\cos (1)\,\cosh (1)\,\sinh (1).
\]
Then the two bracketings are
\[
  \bigl(\exp(p_1)\sin(p_2)\bigr)\cos(p_3)
  =\frac{S}{\sqrt{2}}-\frac{S}{2}i-\frac{S}{2}j
   +\frac{T}{\sqrt{2}}k+\frac{T}{\sqrt{2}}\ell,
\]
whereas
\[
  \exp(p_1)\bigl(\sin(p_2)\cos(p_3)\bigr)
  =\frac{S}{\sqrt{2}}+\frac{S}{2}j
   +\left(\frac{T}{\sqrt{2}}+\frac{S}{2}\right)k
   +\frac{T}{\sqrt{2}}\ell.
\]
Their difference is exactly
\[
  -S\left(\frac{1}{2}i+j+\frac{1}{2}k\right).
\]
This example illustrates that the behaviour of analytic functions of one Pentarnion is the same as when working with complex numbers and complex calculus, while analytic functions of independent Pentarnions display the non-associative behaviour of this highly nonassociative algebra.

\section{Complex-slice Riccati-style equations}\label{sec:complex-slice}

\subsection{Equivariance and invariant complex slice}
In the following, polynomial maps will be ``fully bracketed'' without always referring to this fact, as the nonassociativity of the product of the algebra $\Pent$ usually does not allow us to drop brackets.  Recall that $\mathbf a=\frac{i+j+k}{\sqrt3}$.

\begin{proposition} \label{prop:complex-slice}
Let \(H\) be a fully bracketed polynomial map whose coefficients belong to
\(\Pent^{\Aut(\Pent)}\).  Then
\[
 H(S_a (p))=S_a (H(p))
\]
for all $a\in \R$.
In particular, a Riccati-type equation with coefficients in
\(\mathbb C_{\mathbf a}\) is \(SO(2)\)-equivariant.  If its initial value
also belongs to \(\mathbb C_{\mathbf a}\), the entire solution remains in
that complex subalgebra, and all bracketings coincide.
\end{proposition}

\begin{proof}
The equivariance assertion follows by structural induction on the fully
bracketed expression tree, using that every coefficient is fixed and that
\(S_a\) is an automorphism. 

Since
\(\mathbb C_{\mathbf a}\cong \mathbb{C}\), the vector field is
tangent to \(\mathbb C_{\mathbf a}\) whenever the state lies in it.  Uniqueness of solutions then
implies invariance.  
 Since the slice is isomorphic to $\C$, bracketing is irrelevant here.
\end{proof}

Here, $p$ is viewed a function defined on $\Pent\cong\R^5$ and defined through this equation.

The proposition gives the associative control case.  In the remaining examples the state, the coefficients, or both leave a common complex slice.  The associator is then active, and different parenthesizations need not generate the same vector field.

\subsection{An exact circular example with a complex-slice coefficient}\label{sec:circular-example}
The following example provides a bridge between the associative complex-slice
example and genuinely off-slice Pentarnion dynamics. We demonstrate how a coefficient in
\(\mathbb C_{\mathbf a}\) supports an exact periodic circle for one
bracketing, while the associator generated by the alternative bracketing
imposes a monotone drift that forbids periodic recurrence in \(\rho>0\).
Indeed, the complex-slice result in Proposition \ref{prop:complex-slice} becomes more
informative when the coefficient remains in
\(\mathbb C_{\mathbf a}\) but the state is allowed to leave that slice.
For $\eta\in \R$, $0\leq\eta\leq1$, consider the interpolation 
\begin{equation}
\label{eq:circular-interpolation}
 \dot p=(1-\eta)\mathbf a(pp)+\eta(\mathbf a p)p
       =\mathbf a(pp)+\eta\assoc{\mathbf a}{p}{p}.
\end{equation}

\begin{proposition}[Exact circular periodic family]
\label{prop:exact-circle}
At  \(\eta=0\), the equation
\begin{equation}
\label{eq:left-circle-equation}
 \dot p=\mathbf a(pp)
\end{equation}
has the exact periodic solutions
\begin{equation}
\label{eq:exact-circle}
 p_{r,\theta_0}(t)=r+r\cos(\sqrt2\,rt+\theta_0)\,\mathbf b
   +r\sin(\sqrt2\,rt+\theta_0)\,\mathbf c,
 \qquad r>0.
\end{equation}
Their period is
\begin{equation}
\label{eq:exact-circle-period}
 T_r=\frac{\sqrt2\,\pi}{r}.
\end{equation}
\end{proposition}

\begin{proof}
Set
\[
 q(t)=\cos\theta(t)\,\mathbf b+\sin\theta(t)\,\mathbf c,
 \qquad \theta(t)=\sqrt2\,rt+\theta_0.
\]
By Lemma  \ref{lem:adapted-products}, \(q^2=-1\), so for
\(p=r(1+q)\),
\[
 p^2=r^2(1+q)^2=2r^2q.
\]
Moreover, we observe that
\[
 \mathbf a q
 =\frac{-\sin\theta\,\mathbf b+\cos\theta\,\mathbf c}{\sqrt2}.\]
Consequently, we get
\[
 \mathbf a(pp)
 =\sqrt2\,r^2
  (-\sin\theta\,\mathbf b+\cos\theta\,\mathbf c)
 =\dot p.
\]
The angular velocity  is \(\sqrt2\,r\), which gives
\eqref{eq:exact-circle-period}.
\end{proof}

To understand what the nonassociativity of the Pentarnion algebra means, that is what happens if we change the bracketing in an equation, write
\begin{equation}
\label{eq:adapted-state-circle}
 p=x_0+x_1\mathbf a+x_2\mathbf b+x_3\mathbf c+x_4\mathbf d,
 \qquad \rho=x_2^2+x_3^2.
\end{equation}
\begin{theorem}[Exact associator drift and loss of periodicity]
\label{thm:circle-drift}
For Equation \eqref{eq:circular-interpolation}, direct expansion in the adapted
basis gives
\begin{equation}
\label{eq:circle-drift-components}
 \dot\rho=-\eta x_1\rho,
 \qquad
 \dot x_4=-\eta x_1x_4-\frac{\sqrt3}{2}\eta\rho.
\end{equation}
Hence, on every interval on which \(\rho>0\),
\begin{equation}\label{eq:circle-ratio-drift}
 \frac{d}{dt}\left(\frac{x_4}{\rho}\right)
 =-\frac{\sqrt3}{2}\eta.
\end{equation}
If \(\eta\neq0\), Equation \eqref{eq:circular-interpolation} has no
periodic orbit contained in the region \(\rho>0\).  In particular, the
circular family in Proposition \ref{prop:exact-circle} cannot persist under a nonzero
associator weight.
\end{theorem}

\begin{proof}
Using Lemma  \ref{lem:adapted-products} to expand
\(\mathbf a (pp)+\eta\assoc{\mathbf a}{p}{p}\) gives
\eqref{eq:circle-drift-components}.  The quotient rule then yields
\[
 \frac{d}{dt}\left(\frac{x_4}{\rho}\right)
 =\frac{(-\eta x_1x_4-\frac{\sqrt3}{2}\eta\rho)\rho
        -x_4(-\eta x_1\rho)}{\rho^2}
 =-\frac{\sqrt3}{2}\eta.
\]
For \(\eta\neq0\), the ratio \(x_4/\rho\) is a nonconstant affine
function of time and therefore cannot be periodic.
 
\end{proof}

\section{Compact Pentarnion encoding and rank-four Riccati algebras}
\label{sec:rank-four-riccati}

A Pentarnion differential equation can also be viewed canonically as a structured representation of a polynomial vector field on \(\R^5\). Once a basis is fixed, the equation can be written uniquely as five polynomial equations. The representation is generally much shorter than its coordinate expansion, but it is also restrictive: the coordinate coefficients are tied together by the multiplication table and by the placement of parentheses. Thus the Pentarnion formalism is best viewed as a compact generative language for distinguished families of five-dimensional polynomial fields, rather than as a relabelling of arbitrary coordinate systems. In particular, not every polynomial vector field on \(\R^5\) admits a Pentarnion representation.

Markus initiated the correspondence between homogeneous quadratic differential equations and commutative nonassociative algebras \cite{Markus1960}. R\"ohrl developed related results in \cite{Roehrl1977}. Further developments appear in \cite{KinyonSagle1995}, and the monograph \cite{Walcher1991} treats the subject systematically. Walcher calls a commutative algebra a \emph{rank-three algebra} when each one-generated subalgebra has dimension at most two \cite{Walcher1999}. Algebraic formulations of Riccati systems also occur in Jordan-pair settings \cite{Walcher1991}. Riccati equations over normed division algebras are treated in \cite{deLucasTobolskiVilarino2016}.

Here we apply the quadratic-ODE correspondence only after the nonassociative Pentarnion word has been fixed. We classify all quadratic words containing one coefficient \(B\) and two copies of the state \(p\) whose values are universally confined to \(\operatorname{span}_{\R}\{1,B,p\}\). The term \(p(Bp)-(pB)p=-[p,B,p]\), which vanishes in every flexible algebra, is the flexibility defect. We prove that adding this term generates a third direction: the associated polarised ODE algebra changes generically from rank three, where one-generated subalgebras have dimension at most two, to rank four, where a generic one-generated subalgebra has dimension three. On a canonical active leaf the resulting three-dimensional system admits two first integrals. A nonlinear periodic member is then integrated by quadrature, and its resultant curve is shown to be elliptic.

In this section, a \emph{homogeneous Riccati word} means a homogeneous quadratic expression built from one fixed coefficient and two copies of the dependent variable.  The bracket-complete class below is broader than the two-word interpolation \(B p^2\leftrightarrow(Bp)p\) studied later: it is obtained by allowing every ordering and bracketing of the word \(Bpp\), and then classifying the combinations that force a universal three-dimensional reduction.  Thus no claim is made that a generic member of the later two-word family has rank four.

\subsection{The bracket-complete quadratic word space}
\label{subsec:bracket-complete-space}

Write a Pentarnion in scalar--pure form as \(p=(s,u)\in\R\oplus\R^4\).  The canonical product has the form
\begin{equation}
\label{eq:scalar-vector-product}
 (s,u)(t,v)
 =\bigl(st-\langle u,v\rangle,\;sv+tu+u\times v\bigr),
\end{equation}
Let \(B=(b,w)\) and define the six quadratic words
\begin{equation}
\label{eq:six-bracket-words}
\begin{aligned}
 M_0&=B(pp),& M_1&=(pp)B,\\
 M_2&=(Bp)p,& M_3&=p(Bp),\\
 M_4&=(pB)p,& M_5&=p(pB).
\end{aligned}
\end{equation}
For \(\lambda=(\lambda_0,\ldots,\lambda_5)\in\R^6\), set
\begin{equation}
\label{eq:general-bracket-word-field}
 Q_{\lambda,B}(p)=\sum_{r=0}^{5}\lambda_rM_r(B,p,p).
\end{equation}

\begin{lemma}[Exact six-word expansion]
\label{lem:six-word-expansion}
Put
\[
 U=\langle u,u\rangle,
 \qquad m=\langle w,u\rangle,
 \qquad K=w\times u,
 \qquad T=\langle K,u\rangle,
 \qquad H=K\times u,
\]
and
\[
 A_0=b(s^2-U)-2sm,
 \qquad
 V_0=2bsu+(s^2-U)w,
 \qquad
 V_1=(2bs-m)u+s^2w.
\]
Then, relative to the scalar--pure decomposition \(\Pent=\R1\oplus V\), where \(V=\operatorname{span}_{\R}\{i,j,k,\ell\}\),
\begin{equation}
\label{eq:six-word-expansion}
\begin{aligned}
 M_0&=(A_0,V_0+2sK),&
 M_1&=(A_0,V_0-2sK),\\
 M_2&=(A_0-T,V_1+2sK+H),&
 M_3&=(A_0-T,V_1-H),\\
 M_4&=(A_0+T,V_1-H),&
 M_5&=(A_0+T,V_1-2sK+H).
\end{aligned}
\end{equation}
In particular,
\begin{equation}
\label{eq:six-word-zero-relation}
 M_0-M_1-M_2+M_3-M_4+M_5=0.
\end{equation}
\end{lemma}

\begin{proof}
The square is \(p^2=(s^2-U,2su)\), while
\[
 Bp=(bs-m,bu+sw+K),
 \qquad
 pB=(bs-m,bu+sw-K).
\]
Substitution into Equation~\eqref{eq:scalar-vector-product}, using that \(\times\) is alternating, gives the six displayed identities.  Their stated linear combination vanishes componentwise.
\end{proof}

\begin{theorem}[Universal three-plane classification]
\label{thm:universal-three-plane-classification}
The inclusion
\begin{equation}
\label{eq:universal-three-plane-inclusion}
 Q_{\lambda,B}(p)\in\operatorname{span}_{\R}\{1,B,p\}
 \qquad\text{for every }B,p\in\Pent
\end{equation}
holds if and only if
\begin{equation}
\label{eq:closure-condition-one}
 \lambda_0-\lambda_1+\lambda_2-\lambda_5=0
\end{equation}
and
\begin{equation}
\label{eq:closure-condition-two}
 \lambda_2-\lambda_3-\lambda_4+\lambda_5=0.
\end{equation}
Modulo Equation~\eqref{eq:six-word-zero-relation}, every such field has a unique representation
\begin{equation}
\label{eq:FGH-family}
 Q_{\alpha,\beta,\gamma;B}
 =\alpha\mathcal F_B+\beta\mathcal G_B+\gamma\mathcal H_B,
\end{equation}
where
\begin{equation}
\label{eq:FGH-definitions}
\begin{aligned}
 \mathcal F_B(p)&=B(pp)+(pp)B,\\
 \mathcal G_B(p)&=(Bp)p+(pB)p-B(pp),\\
 \mathcal H_B(p)&=p(Bp)-(pB)p=-[p,B,p].
\end{aligned}
\end{equation}
Consequently, for every initial state \(p_0\), the three-dimensional plane
\begin{equation}
\label{eq:invariant-generated-plane}
 S_{B,p_0}=\operatorname{span}_{\R}\{1,B,p_0\}
\end{equation}
is invariant under \(\dot p=Q_{\alpha,\beta,\gamma;B}(p)\).
\end{theorem}

\begin{proof}
By Lemma~\ref{lem:six-word-expansion}, the pure part of \(Q_{\lambda,B}\) is a linear combination of \(u,w,K,H\).  Its coefficients in the two complementary directions are
\[
 2s(\lambda_0-\lambda_1+\lambda_2-\lambda_5)
 \quad\text{and}\quad
 \lambda_2-\lambda_3-\lambda_4+\lambda_5.
\]
Thus Equations~\eqref{eq:closure-condition-one} and~\eqref{eq:closure-condition-two} are sufficient.  For necessity, take
\[
 w=i+j+k+\ell,
 \qquad
 u=i+j+\ell,
 \qquad s=1.
\]
Then
\[
 w\times u=-\sqrt2\,i,
 \qquad
 (w\times u)\times u=j-\ell,
\]
and the four pure vectors \(w,u,w\times u,(w\times u)\times u\) have determinant \(-2\sqrt2\).  Hence both complementary coefficients must vanish.

The coefficient vectors of \(\mathcal F_B,\mathcal G_B,\mathcal H_B\), together with the coefficient vector of Equation~\eqref{eq:six-word-zero-relation}, are linearly independent and span the four-dimensional solution space of Equations~\eqref{eq:closure-condition-one} and~\eqref{eq:closure-condition-two}.  This proves the representation statement.  Finally, Equation~\eqref{eq:universal-three-plane-inclusion} implies that the vector field maps every point of \(S_{B,p_0}\) back into that subspace, so uniqueness of solutions proves invariance.
\end{proof}

\begin{remark}[Compact encoding of five-dimensional polynomial fields]
\label{rem:compact-five-dimensional-encoding}
A general homogeneous quadratic vector field on \(\R^5\) has \(5\binom{6}{2}=75\) coordinate coefficients.  The universally three-dimensional family in Equation~\eqref{eq:FGH-family} is encoded by the five coordinates of \(B\) and three word coefficients.  It is not the full 75-parameter space; rather, it is a structured eight-parameter family whose coordinate couplings are generated automatically by the Pentarnion multiplication.  The same compression principle is used later with real-linear coefficient maps to encode dense five-dimensional fields in a small number of ordered Pentarnion products.
\end{remark}

\subsection{One-generator rank and the nonflexibility rank jump}
\label{subsec:rank-jump}

For a homogeneous quadratic map \(Q:\Pent\to\Pent\), let  \(\star:\Pent\times \Pent \to\Pent\),
\begin{equation}
\label{eq:polarised-product}
 x\star y=\frac12\bigl(Q(x+y)-Q(x)-Q(y)\bigr)
\end{equation}
be its linearisation, which is a symmetric bilinear map.
Then \((\Pent,\star)\) is a commutative, generally nonassociative algebra and the ODE is \(\dot p=p\star p\).  Following the rank-three convention of Walcher \cite{Walcher1999}, we say that this ODE algebra has \emph{one-generator rank at most \(r\)} if every subalgebra generated by one element has dimension at most \(r-1\).  It has generic rank \(r\) if equality holds on a nonempty Zariski-open set.

Define
\begin{equation}
\label{eq:cNT-definitions}
 c_B(p)=\operatorname{Re}(Bp)=bs-\langle w,u\rangle,
 \qquad
 N(p)=s^2+\langle u,u\rangle,
 \qquad
 T_B(u)=\langle w\times u,u\rangle.
\end{equation}

\begin{proposition}[Compact rank-three components]
\label{prop:compact-GH-identities}
For every \(B=(b,w)\) and \(p=(s,u)\),
\begin{equation}
\label{eq:compact-G-identity}
 \mathcal G_B(p)=2c_B(p)p-N(p)\overline B
\end{equation}
and
\begin{equation}
\label{eq:compact-H-identity}
 \mathcal H_B(p)=-2T_B(u)1.
\end{equation}
Hence the ODE algebras induced separately by \(\mathcal G_B\) and \(\mathcal H_B\) have one-generator rank at most three.
\end{proposition}

\begin{proof}
The identities follow by adding the corresponding rows of Equation~\eqref{eq:six-word-expansion}.  Polarisation of Equation~\eqref{eq:compact-G-identity} maps \(\operatorname{span}\{p,\overline B\}\) into itself, while polarisation of Equation~\eqref{eq:compact-H-identity} maps \(\operatorname{span}\{p,1\}\) into itself.
\end{proof}

\begin{theorem}[Associator-induced rank four]
\label{thm:associator-induced-rank-four}
Let \(B\notin\R\), and let \(\star\) be the polarised product associated with
\begin{equation}
\label{eq:rank-jump-family}
 Q_B(p)=\beta\mathcal G_B(p)+\gamma\mathcal H_B(p).
\end{equation}
Then:
\begin{enumerate}
 \item if \(\beta\gamma=0\), the ODE algebra has one-generator rank at most three;
 \item if \(\beta\gamma\ne0\), the ODE algebra has generic one-generator rank four.
\end{enumerate}
\end{theorem}

Thus adding the flexibility defect \(-[p,B,p]\) to the rank-three field \(\mathcal G_B\) activates a third generated dimension.

\begin{proof}
Only the second assertion requires proof.  Let
\[
 R_B(u)=\langle w\times u,w\rangle,
 \qquad c=c_B(p),\quad N=N(p),\quad T=T_B(u).
\]
Polarisation of Equations~\eqref{eq:compact-G-identity} and~\eqref{eq:compact-H-identity} gives
\[
 x\star y
 =\beta\bigl(c_B(x)y+c_B(y)x-\langle x,y\rangle\overline B\bigr)
 -2\gamma\tau_B(x,y)1,
\]
where \(\tau_B\) is the symmetric polarisation of \(T_B\).  In particular,
\[
 p^{\star2}=2\beta cp-\beta N\overline B-2\gamma T1.
\]
A second multiplication gives
\[
 p^{\star3}=A p+D\overline B+E1,
\]
where, writing \(M=N(B)\),
\begin{align*}
 A&=4\beta^2c^2-\beta^2NM-2\beta\gamma bT,\\
 D&=-2\beta^2cN+2\beta\gamma sT,\\
 E&=-6\beta\gamma cT-\beta\gamma NR_B.
\end{align*}
Therefore we obtain
\begin{equation}
\label{eq:rank-jump-determinant}
\begin{aligned}
 p\wedge p^{\star2}\wedge p^{\star3}
 ={}&\Delta_{B,\beta,\gamma}(p)\,
 p\wedge\overline B\wedge1,\\
 \Delta_{B,\beta,\gamma}(p)
 ={}&\beta\gamma\Bigl[
 \beta\bigl(N^2R_B+2cNT\bigr)+4\gamma sT^2
 \Bigr].
\end{aligned}
\end{equation}
In canonical pure coordinates \(w=(w_1,w_2,w_3,w_4)\) and \(u=(u_1,u_2,u_3,u_4)\), we get
\begin{equation}
\label{eq:nonflexibility-quadratic-form}
 T_B(u)=\sqrt2\Bigl[
 w_4(u_1u_2+u_1u_3+u_2u_3)
 -w_2u_1u_4-w_3u_2u_4-w_1u_3u_4
 \Bigr].
\end{equation}
This polynomial vanishes identically only when \(w=0\).  Choose \(u\) with \(T_B(u)\ne0\) and regard the bracketed expression in Equation~\eqref{eq:rank-jump-determinant} as a polynomial in \(s\).  If \(R_B(u)\ne0\), its leading coefficient is \(\beta R_B(u)\).  If \(R_B(u)=0\) and \(b\ne0\), its cubic coefficient is \(2\beta bT_B(u)\).  If \(R_B(u)=b=0\), it equals
\[
 -2\beta\langle w,u\rangle T_B(u)(s^2+\langle u,u\rangle)
 +4\gamma sT_B(u)^2,
\]
which is again nonzero.  Hence \(\Delta_{B,\beta,\gamma}\) is not the zero polynomial.  On the nonempty Zariski-open set where both \(\Delta\ne0\) and \(p\wedge\overline B\wedge1\ne0\), the three elements \(p,p^{\star2},p^{\star3}\) are independent.  The upper bound three follows from Theorem~\ref{thm:universal-three-plane-classification}, completing the proof.
\end{proof}

\begin{corollary}[Flexible-algebra obstruction]
\label{cor:flexible-obstruction}
In every flexible algebra, \([p,B,p]=0\).  Hence the rank-raising term \(\mathcal H_B\) vanishes identically, and the mechanism of Theorem~\ref{thm:associator-induced-rank-four} is unavailable.
\end{corollary}

\subsection{Canonical three-dimensional normal form and first integrals}
\label{subsec:canonical-rank-four-normal-form}

We use the adapted basis introduced in Equation~\eqref{eq:adapted-basis}.  In particular, Lemma~\ref{lem:adapted-products} gives
\begin{equation}
\label{eq:active-adapted-products}
 \mathbf a^2=\mathbf d^2=-1,
 \qquad \mathbf a\mathbf d=-\sqrt2\,\mathbf a,
 \qquad \mathbf d\mathbf a=\sqrt2\,\mathbf a.
\end{equation}
The three-space
\[
 S=\operatorname{span}_{\R}\{1,\mathbf a,\mathbf d\}
\]
is therefore a nonassociative subalgebra of $\Pent$; for example,
\((\mathbf a\mathbf d)\mathbf d=2\mathbf a\), whereas
\(\mathbf a(\mathbf d\mathbf d)=-\mathbf a\).

\begin{theorem}[Canonical rank-four normal form]
\label{thm:canonical-rank-four-normal-form}
Fix \(B=\mathbf a\) and consider
\begin{equation}
\label{eq:alpha-beta-gamma-flow}
 \dot p=\alpha\mathcal F_{\mathbf a}(p)+\beta\mathcal G_{\mathbf a}(p)+\gamma\mathcal H_{\mathbf a}(p).
\end{equation}
For every initial value \(p_0=s_0+y_0\mathbf a+z_0\mathbf d\), the solution remains in \(S\).  Writing
\[
 p=s+y\mathbf a+z\mathbf d,
 \qquad
 \mu=2\alpha+\beta,
 \qquad
 \nu=\beta,
 \qquad
 \chi=-\sqrt2\gamma,
\]
the Pentarnion equation is exactly equivalent to
\begin{equation}
\label{eq:canonical-rank-four-system}
\begin{aligned}
 \dot s&=-2y(\mu s+\chi z),\\
 \dot y&=\mu(s^2-y^2)+(2\nu-\mu)z^2,\\
 \dot z&=-2\nu yz.
\end{aligned}
\end{equation}
The exact one-generator rank on this leaf is classified in Corollary~\ref{cor:canonical-leaf-rank-classification}.
\end{theorem}

\begin{proof}
For \(p=s+y\mathbf a+z\mathbf d\), Equation~\eqref{eq:active-adapted-products} gives
\begin{align*}
 \mathcal F_{\mathbf a}(p)&=-4sy+2(s^2-y^2-z^2)a,\\
 \mathcal G_{\mathbf a}(p)&=-2sy+(s^2-y^2+z^2)a-2yz\,d,\\
 \mathcal H_{\mathbf a}(p)&=2\sqrt2\,yz.
\end{align*}
Their stated linear combination is Equation~\eqref{eq:canonical-rank-four-system}.
\end{proof}

\begin{corollary}[Exact rank classification on the canonical leaf]
\label{cor:canonical-leaf-rank-classification}
Let \(Q_{\mu,\nu,\chi}\) denote the quadratic map in Equation~\eqref{eq:canonical-rank-four-system}, and let \(\star\) be its polarisation.  For \(v=(s,y,z)^T\), put \(v^{\star2}=Q_{\mu,\nu,\chi}(v)\) and \(v^{\star3}=v\star v^{\star2}\).  Then
\begin{equation}
\label{eq:canonical-leaf-rank-determinant}
 \det\bigl[v,v^{\star2},v^{\star3}\bigr]
 =z\bigl((\mu-\nu)s+\chi z\bigr)\,\Pi_{\mu,\nu,\chi}(s,y,z),
\end{equation}
where
\begin{align}
\label{eq:canonical-leaf-Pi}
 \Pi_{\mu,\nu,\chi}={}&
 \mu^2s^4+2\mu(\mu-\nu)s^2y^2
 +2\mu(2\nu-\mu)s^2z^2+4\mu\chi sy^2z\notag\\
 &+\mu(\mu-2\nu)y^4
 +2(\mu-\nu)(\mu-2\nu)y^2z^2
 +(\mu-2\nu)^2z^4.
\end{align}
The induced algebra has generic one-generator rank four unless
\begin{equation}
\label{eq:canonical-rank-three-strata}
 (\mu=\nu,\ \chi=0)
 \qquad\text{or}\qquad
 (\mu=\nu=0).
\end{equation}
On the two exceptional strata it has one-generator rank at most three.  In particular, \(\beta\gamma\ne0\) in Equation~\eqref{eq:alpha-beta-gamma-flow} implies generic rank four on \(S\).
\end{corollary}

\begin{proof}
Polarising the three coordinate polynomials in Equation~\eqref{eq:canonical-rank-four-system} and taking the determinant gives Equations~\eqref{eq:canonical-leaf-rank-determinant} and~\eqref{eq:canonical-leaf-Pi}.  This determinant is the zero polynomial precisely when its linear factor is identically zero, which gives \(\mu=\nu,\chi=0\), or when \(\Pi\) is identically zero.  In the latter case the coefficient of \(s^4\) gives \(\mu=0\), and the coefficient of \(z^4\) then gives \(\nu=0\).  Outside these cases a nonempty Zariski-open set has nonzero determinant and hence generates the full three-space.  The first exceptional stratum is the pure \(\mathcal G_{\mathbf a}\) family, and the second is the pure \(\mathcal H_{\mathbf a}\) family; Proposition~\ref{prop:compact-GH-identities} gives the rank-three upper bound in both cases.
\end{proof}

\begin{theorem}[Complete integrability on a generic rank-four leaf]
\label{thm:generic-rank-four-integrability}
Assume
\begin{equation}
\label{eq:generic-integrability-conditions}
 z\ne0,
 \qquad
 m:=\nu-\mu\ne0,
 \qquad
 \nu\ne0,
 \qquad
 \delta:=2\nu-\mu\ne0.
\end{equation}
On every connected component where
\[
 L(x):=mx-\chi\ne0,
 \qquad x=\frac{s}{z},
 \qquad q=\frac{y}{z},
\]
system Equation~\eqref{eq:canonical-rank-four-system} has the two first integrals
\begin{equation}
\label{eq:generic-first-integral-K}
 K=z|L(x)|^{\nu/m}
\end{equation}
and
\begin{equation}
\label{eq:generic-first-integral-J}
 J=
 \left(
 q^2+x^2-\frac{2\chi}{\nu}x
 +1+\frac{2\chi^2}{\nu\delta}
 \right)|L(x)|^{-\delta/m}.
\end{equation}
They are functionally independent on a nonempty open subset.  Moreover, every regular trajectory is reduced to the quadrature
\begin{equation}
\label{eq:generic-rank-four-quadrature}
 t-t_0
 =\pm\frac{1}{2K}
 \int_{x_0}^{x}
 \frac{|L(\xi)|^{\nu/m-1}}
 {\sqrt{\mathcal R_J(\xi)}}\,d\xi,
\end{equation}
where
\begin{equation}
\label{eq:effective-radical}
 \mathcal R_J(x)
 =J|L(x)|^{\delta/m}
 -x^2+\frac{2\chi}{\nu}x
 -1-\frac{2\chi^2}{\nu\delta}.
\end{equation}
In particular, if one connected component of \(\{\mathcal R_J\ge0\}\) is a compact interval \([x_-,x_+]\) lying away from \(L=0\), with simple endpoint zeros and \(\mathcal R_J>0\) in its interior, then the corresponding common level of \((J,K)\) is a bounded periodic orbit.  Its period is
\begin{equation}
\label{eq:generic-period-formula}
 T=\frac{1}{|K|}
 \int_{x_-}^{x_+}
 \frac{|L(x)|^{\nu/m-1}}{\sqrt{\mathcal R_J(x)}}\,dx.
\end{equation}
\end{theorem}

\begin{proof}
Use the projective time \(d\tau=z\,dt\).  A direct calculation gives
\begin{equation}
\label{eq:projective-rank-four-system}
 \frac{dx}{d\tau}=2qL(x),
 \qquad
 \frac{dq}{d\tau}=\mu x^2+\delta(q^2+1),
 \qquad
 \frac{dz}{d\tau}=-2\nu qz.
\end{equation}
Consequently,
\[
 \frac{dL}{d\tau}=2mqL,
 \qquad
 \frac{d}{d\tau}
 \left(
 q^2+x^2-\frac{2\chi}{\nu}x
 +1+\frac{2\chi^2}{\nu\delta}
 \right)
 =2\delta q
 \left(
 q^2+x^2-\frac{2\chi}{\nu}x
 +1+\frac{2\chi^2}{\nu\delta}
 \right).
\]
The logarithmic derivatives cancel exactly in Equations~\eqref{eq:generic-first-integral-K} and~\eqref{eq:generic-first-integral-J}.  Since \(\partial K/\partial z\ne0\), while \(\partial J/\partial q\ne0\) when \(q\ne0\), the integrals are generically independent.  Solving Equations~\eqref{eq:generic-first-integral-K} and~\eqref{eq:generic-first-integral-J} for \(z\) and \(q^2\), and substituting into \(\dot x=2qzL\), gives Equation~\eqref{eq:generic-rank-four-quadrature}.  Under the final hypotheses, the two signs of \(q\) form a compact smooth oval and the vector field does not vanish on it.  The period is twice the travel time from \(x_-\) to \(x_+\), giving Equation~\eqref{eq:generic-period-formula}.
\end{proof}

\subsection{A strongly non-linear elliptic periodic orbit}
\label{subsec:rank-four-elliptic-example}

Take
\begin{equation}
\label{eq:rank-four-example-compact}
 \dot p=\mathcal F_{\mathbf a}(p)+\mathcal G_{\mathbf a}(p)
 +\frac1{\sqrt2}\mathcal H_{\mathbf a}(p),
 \qquad p(0)=1+\mathbf d.
\end{equation}
Here \((\mu,\nu,\chi)=(3,1,-1)\), so Equation~\eqref{eq:canonical-rank-four-system} becomes
\begin{equation}
\label{eq:rank-four-example-system}
\begin{aligned}
 \dot s&=-2y(3s-z),\\
 \dot y&=3(s^2-y^2)-z^2,\\
 \dot z&=-2yz,
\end{aligned}
 \qquad (s,y,z)(0)=(1,0,1).
\end{equation}
The generic fractional-power integrals simplify to rational functions.

\begin{proposition}[Rational first integrals]
\label{prop:rank-four-example-integrals}
On \(z\ne0\), system equation~\eqref{eq:rank-four-example-system} has the first integrals
\begin{equation}
\label{eq:rank-four-example-I1}
 I_1=\frac{2s-z}{z^3}
\end{equation}
and
\begin{equation}
\label{eq:rank-four-example-I2}
 I_2=\frac{s^2+y^2+2sz-z^2}{z^3}.
\end{equation}
They take the values \(I_1=1\) and \(I_2=2\) at the prescribed initial point.
\end{proposition}

\begin{proof}
Direct differentiation with Equation~\eqref{eq:rank-four-example-system} gives \(\dot I_1=\dot I_2=0\).  Equivalently, these are the specialisations of Equations~\eqref{eq:generic-first-integral-K} and~\eqref{eq:generic-first-integral-J} after eliminating their common square root.
\end{proof}

\begin{theorem}[Analytic resultant curve and periodicity]
\label{thm:rank-four-example-periodic-orbit}
The solution of Equation~\eqref{eq:rank-four-example-system} is a bounded periodic orbit of affine span three.  It is the compact real component of
\begin{equation}
\label{eq:rank-four-example-two-surfaces}
 2s-z-z^3=0,
 \qquad
 s^2+y^2+2sz-z^2-2z^3=0.
\end{equation}
Eliminating \(s\) gives the resultant curve
\begin{equation}
\label{eq:rank-four-example-resultant}
 4y^2+z^6+6z^4-8z^3+z^2=0.
\end{equation}
Since \(z>0\), the substitution \(Y=2y/z\) gives the smooth quartic model
\begin{equation}
\label{eq:rank-four-example-quartic}
 Y^2=P(z):=-z^4-6z^2+8z-1
 =(1-z)(z^3+z^2+7z-1).
\end{equation}
Let \(r\) be the unique positive root of \(z^3+z^2+7z-1\).  Then the orbit is parametrised by the compact oval \(r\le z\le1\), and its period is
\begin{equation}
\label{eq:rank-four-example-period}
 T=2\int_r^1\frac{dz}{z^2\sqrt{P(z)}}.
\end{equation}
The birational change
\begin{equation}
\label{eq:rank-four-example-weierstrass-map}
 X=\frac{1+z}{1-z},
 \qquad
 W=\frac{Y}{(1-z)^2}
\end{equation}
identifies the projective closure with the elliptic curve
\begin{equation}
\label{eq:rank-four-example-weierstrass}
 W^2=X^3-X-1.
\end{equation}
\end{theorem}

\begin{proof}
The integral values from Proposition~\ref{prop:rank-four-example-integrals} give Equation~\eqref{eq:rank-four-example-two-surfaces}.  The first equation gives \(s=(z+z^3)/2\); substitution into the second gives Equation~\eqref{eq:rank-four-example-resultant}.  The cubic factor in equation~\eqref{eq:rank-four-example-quartic} is strictly increasing because its derivative \(3z^2+2z+7\) is positive.  It is negative at zero and positive at one, so it has a unique root \(r\in(0,1)\), and \(P(z)>0\) for \(r<z<1\).  The quartic discriminant is \(-94208\), hence the endpoint roots are simple and the two branches \(Y=\pm\sqrt{P(z)}\) form a smooth compact oval.

The lifted curve equation~\eqref{eq:rank-four-example-two-surfaces} is invariant by the two first integrals.  It contains no equilibrium: in its interior \(y\ne0\).  At a turning point, write \(Y=2y/z\).  Along either open branch, differentiating \(Y^2=P(z)\) and using \(\dot z=-Yz^2\) gives \(2\dot Y=-z^2P'(z)\); continuity extends this identity to the endpoints.  Since both endpoint roots are simple, \(\dot Y\ne0\), equivalently \(\dot y\ne0\).  A nonvanishing vector field on a compact connected one-dimensional manifold is periodic.  Furthermore, no affine plane contains the orbit.  Indeed, a plane relation \(As+By+Cz+D=0\) holding on both signs of \(y\) forces \(B=0\); substituting \(s=(z+z^3)/2\) on an interval then forces \(A=C=D=0\).  Thus the affine span is three.

Finally, \(\dot z=-2yz\) and Equation~\eqref{eq:rank-four-example-quartic} give \(\dot z^{\,2}=z^4P(z)\), which yields Equation~\eqref{eq:rank-four-example-period}.  Substitution of Equation~\eqref{eq:rank-four-example-weierstrass-map} into Equation~\eqref{eq:rank-four-example-quartic} gives Equation~\eqref{eq:rank-four-example-weierstrass}; its discriminant is nonzero.
\end{proof}

Numerically,
\begin{equation}
\label{eq:rank-four-example-numerics}
 r=0.139680581996107\ldots,
 \qquad
 T=25.0139586026652\ldots.
\end{equation}
A DOP853 integration over the quadrature period returns to the initial point with Euclidean defect \(6.99\times10^{-12}\); the maximum residuals in \(I_1=1\), \(I_2=2\), and Equation~\eqref{eq:rank-four-example-resultant} are respectively \(3.21\times10^{-13}\), \(6.62\times10^{-14}\), and \(7.08\times10^{-13}\).  The orbit is visibly nonelliptic: its relative least-squares error against the best first-harmonic spatial ellipse is approximately \(0.811\).

\begin{figure}[H]
 \centering
 \includegraphics[width=0.99\linewidth]{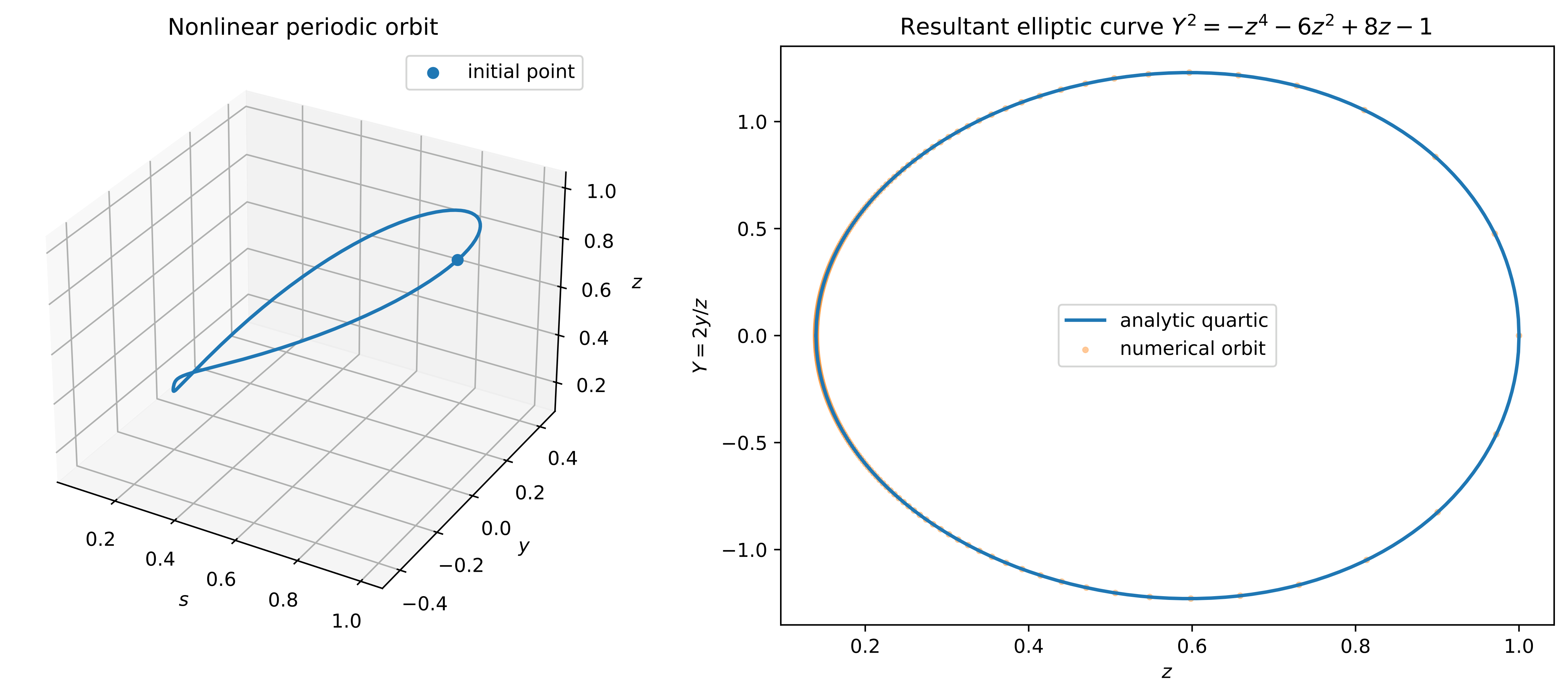}
 \caption{The rank-four periodic solution of Equation~\eqref{eq:rank-four-example-system}.  Left: the three-coordinate orbit in \((s,y,z)\), whose affine span is three.  Right: numerical samples projected onto the exact resultant quartic \(Y^2=-z^4-6z^2+8z-1\).  The numerical trajectory is used only as an illustration; periodicity and the elliptic curve follow from Theorem~\ref{thm:rank-four-example-periodic-orbit}.}
 \label{fig:rank-four-elliptic-orbit}
\end{figure}

\section{Full-dimensional elliptic and bracketing-dependent flows}
\label{sec:pentarnion-odes}

Let
\[
 p=w+xi+yj+zk+u\ell\in\Pent,
 \qquad [p]_{\mathcal B}=(w,x,y,z,u)^{\mathsf T},
 \qquad \mathcal B=(1,i,j,k,\ell),
\]
and recall that $[p,q,r]=(pq)r-p(qr)$
is the associator of $\Pent$.  The preceding section classified bracket-complete homogeneous Riccati words whose induced ODE algebras have one-generator rank four and whose trajectories are confined to invariant three-dimensional leaves.  We now turn to two complementary mechanisms that use the full five-dimensional coordinate space.  The failure of associativity restricts which dense quadratic vector fields can be represented by a small number of ordered products in the fixed algebra, and it makes different parenthesizations of the same formal monomial dynamically inequivalent.

First, we state a quadratic Pentarnion equation and solve the initial-value problem $p(0)=i$ from invariant polynomial relations.  The reduction produces an irreducible genus-one curve, rather than a projective or stereographic Riccati linearisation, and the surrounding five-dimensional field admits no fixed linear block decomposition.  We then keep the coefficients of a second equation fixed and vary only the bracketing between $B(pp)$ and $(Bp)p$.  This isolates the associator while preserving the linearisation at the origin, and makes it possible to compare the resulting recurrent dynamics.  The appendix gives a complementary solvable family in which exact compensation creates a full-span $2{:}3{:}5$ periodic orbit and an uncompensated associator destroys that orbit at the level of its initial tangent.

\subsection{An irreducible quadratic flow with an elliptic solution}
\label{subsec:genus-one-pentarnion}

For a real-linear map \(\mathcal A:\Pent\to\Pent\), write
\([\mathcal A(p)]_{\mathcal B}=A[p]_{\mathcal B}\).  Define \(\mathcal N,\mathcal M,\mathcal K\) by
\begin{equation}
\label{eq:genus-one-matrices}
N=\begin{pmatrix}
0&1&0&-1&0\\
-\frac32&0&1&0&0\\
0&0&0&0&0\\
0&0&0&0&0\\
0&0&0&0&0
\end{pmatrix},
\end{equation}
\[
M=\begin{pmatrix}
0&0&0&0&0\\
-\frac12&0&\frac12&0&0\\
0&1&0&-\frac98&-\frac98\\
-\frac32&0&\frac98&0&-\frac38\\
0&0&\frac98&\frac38&0
\end{pmatrix},
\qquad
K=\begin{pmatrix}
0&0&0&0&0\\
\frac{9\sqrt2}{8}&0&0&0&0\\
0&0&0&0&0\\
0&0&0&0&0\\
0&0&0&0&0
\end{pmatrix}.
\]
We consider
\begin{equation}
\label{eq:compact-genus-one-pentarnion}
 \dot p=k+\mathcal N(p)
 +p\bigl((\mathcal M+\mathcal K)(p)\bigr)
 +\bigl((\mathcal M-\mathcal K)(p)\bigr)p .
\end{equation}
The coefficient maps are fixed endomorphisms of the canonical vector space.  All nonlinear operations in Equation~\eqref{eq:compact-genus-one-pentarnion} are products in the original Pentarnion algebra, and the displayed order of multiplication is part of the equation.

\begin{proposition}[Canonical component system]
\label{prop:genus-one-components}
Equation~\eqref{eq:compact-genus-one-pentarnion} is equivalent to
\begin{equation}
\label{eq:dense-genus-one-flow}
\begin{aligned}
 \dot w&=x-z+wx+3wz-3xy,\\
 \dot x&=y-\frac32w-w^2+wy,\\
 \dot y&=2wx,\\
 \dot z&=1-3w^2+\frac32uw,\\
 \dot u&=3wz.
\end{aligned}
\end{equation}
In particular, every component contains a nonzero quadratic term, and the five equations are coupled through the mixed products $wx$, $wz$, $xy$, $wy$ and $uw$.
\end{proposition}

\begin{proof}
Substitution of the matrices in Equation~\eqref{eq:genus-one-matrices} into Equation~\eqref{eq:compact-genus-one-pentarnion}, followed by expansion with the canonical multiplication table, gives
\[
[p(\mathcal M+\mathcal K)(p)+(\mathcal M-\mathcal K)(p)p]_{\mathcal B}
=
\begin{pmatrix}
wx+3wz-3xy\\
-w^2+wy\\
2wx\\
-3w^2+\frac32uw\\
3wz
\end{pmatrix},
\]
while \([k+\mathcal N(p)]_{\mathcal B}=(x-z,\,y-\frac32w,\,0,\,1,\,0)^{\mathsf T}\).  Adding the two contributions yields Equation~\eqref{eq:dense-genus-one-flow}.
\end{proof}

\subsubsection{Invariant ideal and integration of the initial-value problem}

We solve Equation~\eqref{eq:compact-genus-one-pentarnion} with the canonical initial value
\begin{equation}
\label{eq:genus-one-p0}
 p(0)=i.
\end{equation}
The component equations do not display an immediate scalar reduction.  Integrability instead follows from four polynomial relations that are preserved by the vector field.

\begin{theorem}[Invariant prime genus-one ideal]
\label{thm:invariant-genus-one-ideal}
Define
\begin{equation}
\label{eq:genus-one-generators}
\begin{aligned}
 G_1&=y-w^2, & G_2&=z-wx,\\
 G_3&=u-wy, &
 G_4&=x^2-\frac12y^2+\frac32y-1.
\end{aligned}
\end{equation}
For the vector field \(F\) in equation~\eqref{eq:dense-genus-one-flow},
\begin{equation}
\label{eq:ideal-lie-identities}
\begin{aligned}
 \mathcal L_FG_1
 &=6wxG_1+2w(1-3w)G_2,\\
 \mathcal L_FG_2
 &=\left(3x^2-w-\frac12y+\frac32\right)G_1
   +x(1-3w)G_2+\frac32wG_3-G_4,\\
 \mathcal L_FG_3
 &=x(3y-1)G_1+(y+3w-3wy)G_2,\\
 \mathcal L_FG_4
 &=2xG_1.
\end{aligned}
\end{equation}
Consequently,
\[
I=\langle G_1,G_2,G_3,G_4\rangle
\]
is an invariant prime ideal.  The complexification of \(\mathcal E=V(I)\subset\R^5\) is an irreducible affine curve, and its nonsingular projective completion has genus one.
\end{theorem}

\begin{proof}
The identities in Equation~\eqref{eq:ideal-lie-identities} give \(\mathcal L_F(I)\subset I\), so \(V(I)\) is invariant.  The first three generators eliminate
\[
y=w^2,\qquad z=wx,\qquad u=wy=w^3.
\]
The remaining relation is
\begin{equation}
\label{eq:jacobi-quartic}
 x^2=(1-w^2)\left(1-\frac12w^2\right).
\end{equation}
The polynomial on the right has the four distinct roots \(\pm1,\pm\sqrt2\) and is not a square in \(\C[w]\).  Hence
\[
\C[w,x,y,z,u]/I
\simeq
\C[w,x]/\left\langle x^2-(1-w^2)(1-\tfrac12w^2)\right\rangle
\]
is a domain.  The smooth projective model is a double cover of \(\mathbb P^1\) branched at four points, and the Riemann--Hurwitz formula gives genus one.
\end{proof}

\begin{corollary}[Exact Jacobi-elliptic Pentarnion solution]
\label{cor:jacobi-elliptic-orbit}
The initial-value problem equations~\eqref{eq:compact-genus-one-pentarnion} and~\eqref{eq:genus-one-p0} have the globally regular periodic solution 
\begin{equation}
\label{eq:jacobi-pentarnion-orbit}
 p_*(t)=s(t)+v(t)i+s(t)^2j+s(t)v(t)k+s(t)^3\ell,
\end{equation}
where
\[
 s(t)=\operatorname{sn}\!\left(t\mid\frac12\right),
 \qquad
 v(t)=\operatorname{cn}\!\left(t\mid\frac12\right)
       \operatorname{dn}\!\left(t\mid\frac12\right).
\]
Its minimal period is
\begin{equation}
T=4K\!\left(\frac12\right).
\label{eq:genus-one-period}
\end{equation}
\end{corollary}

\begin{proof}
The initial point lies on \(\mathcal E\).  Restricting Equation~\eqref{eq:dense-genus-one-flow} to \(\mathcal E\) gives
\[
 \dot w=x,
 \qquad
 \dot x=w^3-\frac32w,
\]
and Equation~\eqref{eq:jacobi-quartic} gives
\[
 \dot w^{\,2}=(1-w^2)\left(1-\frac12w^2\right).
\]
The initial values \(w(0)=0\) and \(\dot w(0)=1\) select
\(w(t)=\operatorname{sn}(t\mid1/2)\).  The standard derivative identity for \(\operatorname{sn}\) gives \(x=v\), and \(G_1=G_2=G_3=0\) yields the remaining components.  The real function \(\operatorname{sn}(t\mid m)\) has minimal period \(4K(m)\) for \(0<m<1\).
\end{proof}

\begin{proposition}[Full affine span and elliptic normal sextic]
\label{prop:elliptic-normal-sextic}
The image of Equation~\eqref{eq:jacobi-pentarnion-orbit} has affine span equal to \(\Pent\).  Its projective closure under
\[
 [X_0:X_1:X_2:X_3:X_4:X_5]
 =[1:w:x:w^2:wx:w^3]
\]
is an elliptic normal curve of degree six in \(\mathbb P^5\).  Its homogeneous ideal is generated by nine quadrics.
\end{proposition}

\begin{proof}
An affine relation on the orbit has the form
\[
 a_0+a_1w+a_3w^2+a_5w^3+x(a_2+a_4w)=0.
\]
For every \(w\in(-1,1)\), the compact real oval contains both signs of \(x\).  Subtracting the two relations gives \(a_2+a_4w=0\) on an interval, hence \(a_2=a_4=0\).  The remaining cubic then vanishes on an interval, so all coefficients are zero.

Let \(P_+,P_-\) be the two points at infinity of the smooth projective quartic.  The functions \(w\) and \(x\) have pole orders one and two, respectively, at each point.  Thus
\[
1,w,x,w^2,wx,w^3\in L(3P_++3P_-).
\]
They are linearly independent, and Riemann--Roch gives
\(\dim L(3P_++3P_-)=6\).  Hence they form the complete degree-six linear system and define an elliptic normal sextic in \(\mathbb P^5\).  The ideal of an elliptic normal curve of degree \(N\ge4\) is generated by \(N(N-3)/2\) quadrics; for \(N=6\), this number is nine \cite{KanekoKuwata2024,Shatsila2022}.
\end{proof}

\begin{figure}[H]
\centering
\includegraphics[width=.98\linewidth]{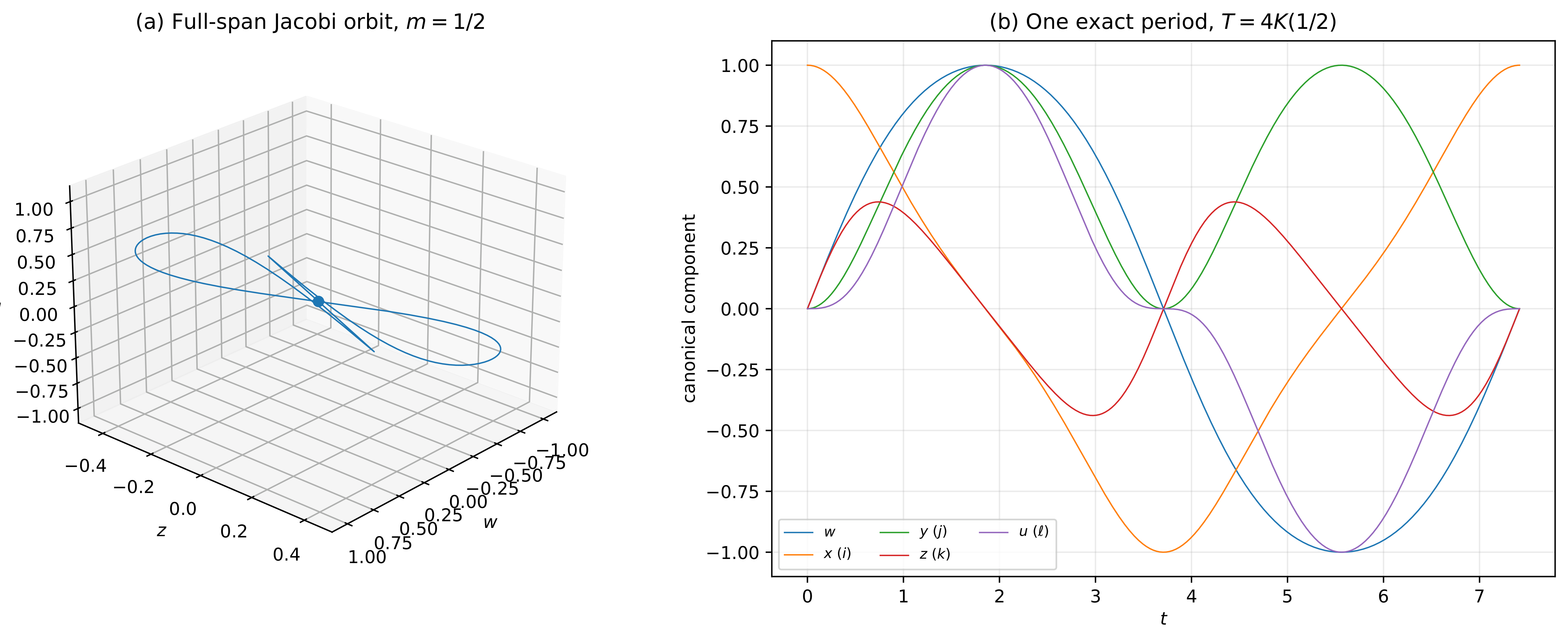}
\caption{The exact genus-one solution of Equation~\eqref{eq:compact-genus-one-pentarnion}.  Left: its projection onto the canonical \((w,z,u)\) coordinates.  Right: the five canonical components over one period \(4K(1/2)\).}
\label{fig:genus-one-exact}
\end{figure}

\subsubsection{Irreducibility and modulus selection}

The one-dimensional invariant curve does not arise from a fixed linear splitting of the surrounding five-dimensional vector field.

\begin{theorem}[Linear irreducibility of the ambient flow]
\label{thm:linear-irreducibility}
Write
\[
 DF(p)=J_0+wJ_w+xJ_x+yJ_y+zJ_z+uJ_u.
\]
The associative matrix algebra generated by \(J_0,J_w,J_x,J_y,J_z,J_u\) is \(M_5(\Q)\).  Consequently, their Jacobian family has no nonzero proper common invariant subspace, and no fixed linear coordinate transformation decomposes or simultaneously block-triangularizes the ambient five-dimensional flow.
\end{theorem}

\begin{proof}
Consider the 25 matrix words
\begin{align*}
\mathcal W=\{&I,J_0,J_w,J_x,J_y,J_z,J_u,
J_0^2,J_0J_w,J_0J_x,J_0J_y,J_0J_z,J_0J_u,\\
&J_wJ_0,J_w^2,J_wJ_x,J_wJ_y,
J_xJ_0,J_xJ_w,J_x^2,
J_zJ_0,J_zJ_w,J_zJ_x,J_0^2J_w,J_xJ_0J_w\}.
\end{align*}
After vectorizing the matrices in this order, the resulting \(25\times25\) matrix has determinant
\[
2541865828329=3^{26}\neq0.
\]
Thus \(\mathcal W\) spans \(M_5(\Q)\).  Any common invariant subspace for the Jacobian coefficients would then be invariant under every matrix in \(M_5(\Q)\), which is impossible except for \(\{0\}\) and \(\Q^5\).
\end{proof}

\begin{proposition}[Exclusion of projective/conformal Riccati forms]
\label{prop:not-riccati}
The homogeneous quadratic part of Equation~\eqref{eq:dense-genus-one-flow} is not linearly equivalent to a projective Riccati field or to a conformal Riccati field associated with any nondegenerate real symmetric bilinear form.
\end{proposition}

\begin{proof}
For a projective Riccati field, the quadratic term has the form \(Q(v)=\lambda(v)v\), and is therefore collinear with \(v\) for every \(v\).  Here
\[
Q(1,0,0,0,0)^{\mathsf T}=(0,-1,0,-3,0)^{\mathsf T},
\]
which is not collinear with the input.  Collinearity of the radial quadratic field is preserved by invertible linear coordinate changes.

For a conformal Riccati field relative to a nondegenerate symmetric matrix \(G\), the derivative matrices of the quadratic part satisfy
\[
 DQ(v)^{\mathsf T}G+GDQ(v)
 =\frac25\operatorname{tr}(DQ(v))G
\]
for every \(v\) \cite{deLucasTobolskiVilarino2016}.  Substitution of the five coefficient matrices gives a homogeneous linear system for the 15 independent entries of \(G\) with rank 15.  Hence \(G=0\) is the only solution, and no nondegenerate conformal metric exists.
\end{proof}

The elliptic modulus is not forced by the geometry of the invariant quartic alone.  To isolate the restriction imposed by Pentarnion multiplication, consider the family
\begin{equation}
\label{eq:jacobi-family}
\begin{aligned}
 \dot w&=x-z+wx+3wz-3xy,\\
 \dot x&=y-(1+m)w-w^2+2mwy,\\
 \dot y&=2wx,\\
 \dot z&=1-2(1+m)w^2+3muw,\\
 \dot u&=3wz.
\end{aligned}
\end{equation}
The ideal generated by \(G_1,G_2,G_3\) and
\[
G_{4,m}=x^2-my^2+(1+m)y-1
\]
is invariant.  On its compact real oval, the solution through \(p(0)=i\) is
\[
p_m(t)=s_m+s_m' i+s_m^2j+s_ms_m'k+s_m^3\ell,
\qquad
s_m=\operatorname{sn}(t\mid m).
\]

\begin{theorem}[Pentarnion selection of the Jacobi modulus]
\label{thm:modulus-selection}
Allow arbitrary real-linear maps \(\mathcal A,\mathcal B,\mathcal N_m:\Pent\to\Pent\).  The Jacobi Family Equation~\eqref{eq:jacobi-family} admits a representation
\begin{equation}
\label{eq:two-product-ansatz}
 \dot p=k+\mathcal N_m(p)+p\mathcal A(p)+\mathcal B(p)p
\end{equation}
using two ordered canonical Pentarnion products if and only if
\[
m=\frac12.
\]
At \(m=1/2\), Equation~\eqref{eq:compact-genus-one-pentarnion} gives an explicit representation.
\end{theorem}

\begin{proof}
Matching the quadratic coefficients in Equation~\eqref{eq:two-product-ansatz} defines a linear map
\[
\mathcal C:\operatorname{End}(\Pent)\oplus\operatorname{End}(\Pent)
\longrightarrow\operatorname{Hom}(S^2\Pent,\Pent).
\]
In canonical coordinates, \(\mathcal C\) is represented by a \(75\times50\) matrix of rank 44.  A basis of its left nullspace applied to the quadratic coefficient vector of equation~\eqref{eq:jacobi-family} produces three nonzero compatibility conditions, all proportional to \(1-2m\).  Therefore the quadratic field belongs to \(\operatorname{im}\mathcal C\) exactly for \(m=1/2\).
\end{proof}

\subsubsection{Relation with quadratic algebras and elliptic differential systems}
The use of nonassociative algebras to encode homogeneous quadratic differential equations was developed systematically in \cite{Markus1960,KinyonSagle1995}, where the multiplication is usually reconstructed from a prescribed quadratic map. Equation~\eqref{eq:compact-genus-one-pentarnion} proceeds in the opposite direction: the canonical Pentarnion multiplication is fixed first, and the quadratic field is required to lie in the image of two ordered product maps. The coefficient calculation in Theorem~\ref{thm:modulus-selection} shows that this representability condition is restrictive enough to select \(m=1/2\) from the continuous Jacobi family in Equation~\eqref{eq:jacobi-family}. Thus the genus-one curve is recovered from, and algebraically constrained by, the fixed Pentarnion equation rather than inserted into an arbitrary quadratic extension.

\subsection{A bracketing family with associator-controlled global dynamics}
\label{subsec:bracketing-family}

We now hold the algebra, coefficients and initial condition fixed and vary only the placement of parentheses.  Let
\begin{align}
A&=-\frac12-2i+\frac12j-k-\frac12\ell,
\label{eq:bracketing-A}\\
B&=1+2j-\frac12\ell.
\label{eq:bracketing-B}
\end{align}
For \(0\le\eta\le1\), consider
\begin{equation}
\label{eq:bracketing-compact}
\dot p=Ap+(1-\eta)B(pp)+\eta(Bp)p
      =Ap+B(pp)+\eta[B,p,p].
\end{equation}
Thus \(\eta\) measures the associator contribution without changing \(A\), \(B\), or the underlying multiplication table.  The endpoints are the left-bracketed Equation \(Ap+B(pp)\) and the right-bracketed Equation \(Ap+(Bp)p\).

All trajectories reported below start from
\begin{equation}
\label{eq:bracketing-p0}
p(0)=p_0=\frac12(1-i-j+k-\ell).
\end{equation}
At this point the deformation is already active:
\begin{equation}
\label{eq:bracketing-initial-associator}
\begin{aligned}
[B,p_0,p_0]
 &=\frac18\left(3\sqrt2+7i+13j-5k-15\ell\right),\\
\|[B,p_0,p_0]\|&=\frac{9\sqrt6}{8}.
\end{aligned}
\end{equation}

\subsubsection{Matrix representation and variational equation}

Let \(x=[p]_{\mathcal B}\).  For a fixed element
\(q=q_0+q_1i+q_2j+q_3k+q_4\ell\), define \(L(q)\) by
\([qp]_{\mathcal B}=L(q)x\).  The canonical multiplication table gives
\begingroup
\small
\begin{equation}
\label{eq:bracketing-left-matrix}
L(q)=
\begin{pmatrix}
q_0&-q_1&-q_2&-q_3&-q_4\\
q_1&q_0&\frac{-q_3+q_4}{\sqrt2}&\frac{q_2+q_4}{\sqrt2}&\frac{-q_2-q_3}{\sqrt2}\\
q_2&\frac{q_3+q_4}{\sqrt2}&q_0&\frac{-q_1+q_4}{\sqrt2}&\frac{-q_1-q_3}{\sqrt2}\\
q_3&\frac{-q_2+q_4}{\sqrt2}&\frac{q_1+q_4}{\sqrt2}&q_0&\frac{-q_1-q_2}{\sqrt2}\\
q_4&\frac{-q_2+q_3}{\sqrt2}&\frac{q_1-q_3}{\sqrt2}&\frac{-q_1+q_2}{\sqrt2}&q_0
\end{pmatrix}.
\end{equation}
\endgroup
With \(C=\operatorname{diag}(1,-1,-1,-1,-1)\), right multiplication is
\begin{equation}
R(q)=C L(\bar q) C,
\qquad
[pq]_{\mathcal B}=R(q)x.
\label{eq:bracketing-right-matrix}
\end{equation}
The square map is
\begin{equation}
Q(x):=[p^2]_{\mathcal B}
=
\begin{pmatrix}
x_0^2-x_1^2-x_2^2-x_3^2-x_4^2\\
2x_0x_1\\2x_0x_2\\2x_0x_3\\2x_0x_4
\end{pmatrix}.
\label{eq:bracketing-square-map}
\end{equation}
Consequently, Equation~\eqref{eq:bracketing-compact} is equivalent to
\begin{equation}
\dot x=L(A)x+(1-\eta)L(B)Q(x)+\eta R(p)L(B)x.
\label{eq:bracketing-matrix-flow}
\end{equation}
In particular,
\[
[B,p,p]_{\mathcal B}=R(p)L(B)x-L(B)Q(x).
\]

\begin{proposition}[Tangent equation, Jacobian and divergence]
\label{prop:bracketing-tangent-divergence}
For a tangent perturbation \(h\),
\begin{equation}
\dot h=Ah+(1-\eta)B(ph+hp)
+\eta\bigl((Bh)p+(Bp)h\bigr).
\label{eq:bracketing-tangent}
\end{equation}
In canonical coordinates,
\begin{equation}
DF_\eta(x)
=L(A)+(1-\eta)L(B)\bigl(L(p)+R(p)\bigr)
+\eta\bigl(R(p)L(B)+L(Bp)\bigr),
\label{eq:bracketing-jacobian}
\end{equation}
and
\begin{equation}
\nabla\!\cdot F_\eta(x)
=-\frac52+10x_0-(8+4\eta)x_2+(2+3\eta)x_4.
\label{eq:bracketing-divergence}
\end{equation}
\end{proposition}

\begin{proof}
The directional derivative of \(p\mapsto pp\) is \(ph+hp\), whereas that of \(p\mapsto(Bp)p\) is \((Bh)p+(Bp)h\).  Translation through \(L\) and \(R\) gives Equation~\eqref{eq:bracketing-jacobian}, and taking the trace gives Equation~\eqref{eq:bracketing-divergence}.
\end{proof}

\subsubsection{Exact local structure at the origin}

Because \(\eta\) appears only in terms that are quadratic in \(p\), changing the bracketing does not change the linearised flow at the origin.

\begin{theorem}[Parameter-independent saddle-focus]
\label{thm:bracketing-origin}
For every \(0\le\eta\le1\), the origin of Equation~\eqref{eq:bracketing-compact} has linearisation \(L(A)\).  Its characteristic polynomial is
\begin{equation}
\chi_{L(A)}(\lambda)=
\left(\lambda^2+\lambda+\frac{23}{4}\right)
\left(\lambda^3+\frac32\lambda^2+3\lambda
+\frac54+\frac{49\sqrt2}{8}\right).
\label{eq:bracketing-charpoly}
\end{equation}
The origin is therefore a hyperbolic saddle-focus with a two-dimensional unstable manifold and a three-dimensional stable manifold for all \(\eta\in[0,1]\).
\end{theorem}

\begin{proof}
Both quadratic terms vanish to second order at the origin, so
\(DF_\eta(0)=L(A)\).  In the canonical basis,
\[
L(A)=\begin{pmatrix}
-\tfrac12&2&-\tfrac12&1&\tfrac12\\
-2&-\tfrac12&\tfrac{\sqrt2}{4}&0&\tfrac{\sqrt2}{4}\\
\tfrac12&-\tfrac{3\sqrt2}{4}&-\tfrac12&\tfrac{3\sqrt2}{4}&\tfrac{3\sqrt2}{2}\\
-1&-\tfrac{\sqrt2}{2}&-\tfrac{5\sqrt2}{4}&-\tfrac12&\tfrac{3\sqrt2}{4}\\
-\tfrac12&-\tfrac{3\sqrt2}{4}&-\tfrac{\sqrt2}{2}&\tfrac{5\sqrt2}{4}&-\tfrac12
\end{pmatrix}.
\]
Its characteristic polynomial factors as in Equation~\eqref{eq:bracketing-charpoly}.  The quadratic factor has roots
\[
-\frac12\pm\frac{\sqrt{22}}{2}\,\mathrm{i}_{\mathrm c},
\]
where \(\mathrm{i}_{\mathrm c}^2=-1\) is the ordinary complex unit.  The first column of the Routh array for the cubic factor has signs \(+,+,-,+\), so the cubic has two roots in the open right half-plane and one in the open left half-plane.  Numerically they are
\[
0.346511\pm2.097557\,\mathrm{i}_{\mathrm c},
\qquad
-2.193023.
\]
Thus no eigenvalue lies on the imaginary axis, and the dimensions of the stable and unstable manifolds follow.
\end{proof}

\raggedbottom
\subsubsection{Global dynamics at the two bracketings}

The same initial value equation~\eqref{eq:bracketing-p0} is used at both endpoints.  Fixed-step fourth-order Runge--Kutta calculations used \(\Delta t=0.005\), a transient of \(1000\), a \(5000\)-time-unit Lyapunov window and QR orthonormalization every \(0.25\) time units.  The phase portraits were plotted without point thinning.  Independent adaptive DOP853 integrations were used to cross-check the largest Lyapunov exponent.

At \(\eta=0\), the trajectory remains bounded over the computed horizon and exhibits irregular recurrent bursts.  At \(\eta=1\), it converges to a smooth periodic orbit.  The full finite-time spectra are
\begin{table}[H]
\centering
\resizebox{\linewidth}{!}{\begin{tabular}{crrrrrrr}
\toprule
$\eta$ & $\lambda_1$ & $\lambda_2$ & $\lambda_3$ & $\lambda_4$ & $\lambda_5$ & $\sum\lambda_i$ & $\langle\nabla\!\cdot F\rangle$\\
\midrule
0 & 0.351317 & -0.000065 & -0.588202 & -1.224216 & -2.427302 & -3.888468 & -3.888468\\
1 & -0.000025 & -0.110993 & -0.324610 & -0.736304 & -2.000981 & -3.172913 & -3.172913\\
\bottomrule
\end{tabular}}
\caption{Finite-time Lyapunov spectra and the independent mean-divergence check for Equation~\eqref{eq:bracketing-p0}.}
\label{tab:bracketing-spectra}
\end{table}
The adaptive calculation gave \(\lambda_1\approx0.33579\) at \(\eta=0\) and \(\lambda_1\approx-2.25\times10^{-4}\) at \(\eta=1\).  At the left-bracketed endpoint, the Kaplan--Yorke dimension estimate is
\[
D_{KY}\approx2.5972.
\]
Timestep refinement gave
\[
(\Delta t,\lambda_1)=(0.01,0.3433),\quad(0.005,0.3533),\quad(0.0025,0.3586),
\]
while the corresponding values at \(\eta=1\) remained numerically zero.

\begin{figure}[!t]
\centering
\includegraphics[width=.99\linewidth]{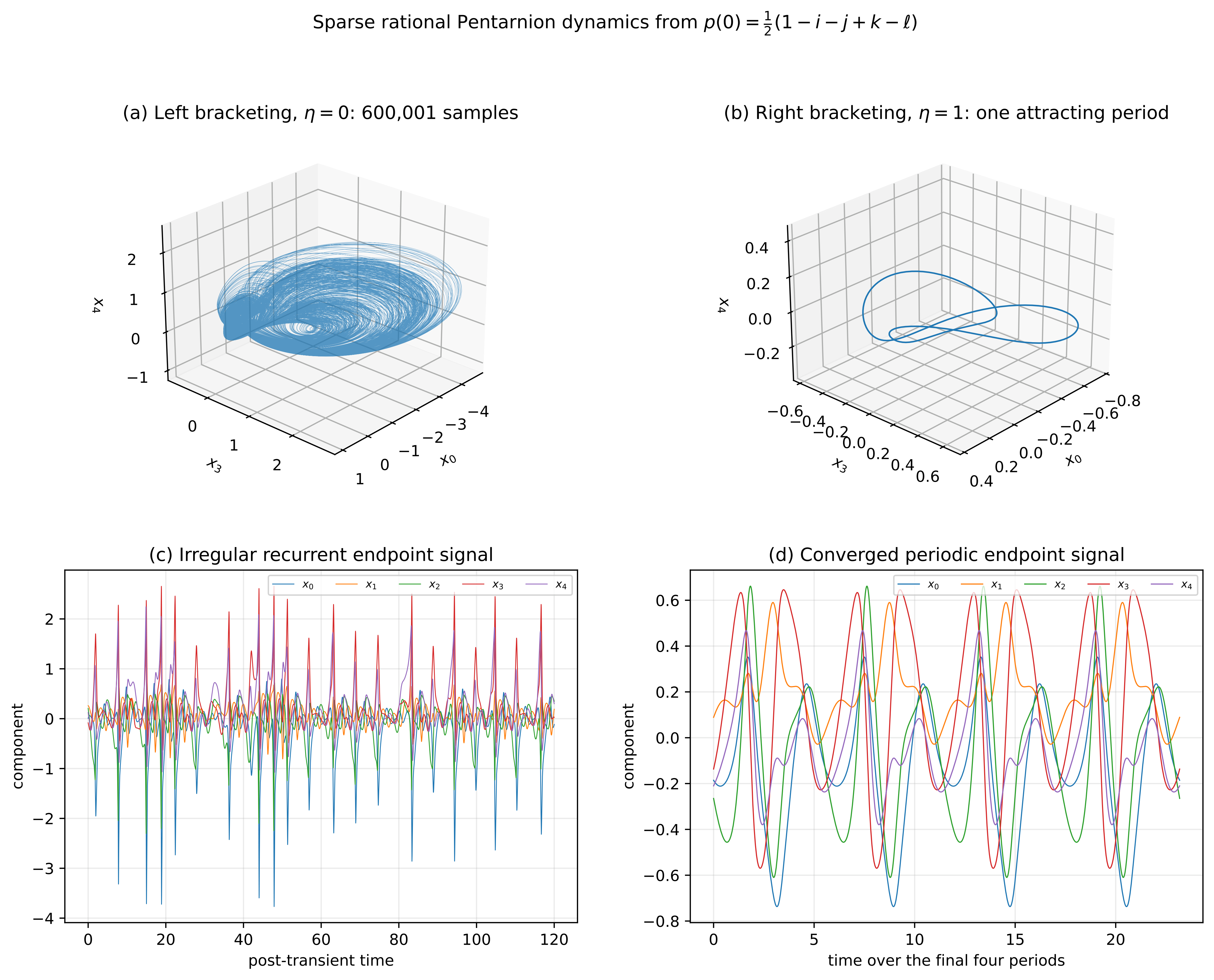}
\caption{Endpoint dynamics of Equation~\eqref{eq:bracketing-compact}.  Left: the bounded aperiodic regime at \(\eta=0\), for which the diagnostics below give numerical evidence of chaos.  Right: convergence to an attracting periodic orbit at \(\eta=1\).  The lower panels show the five canonical components using the full recorded sampling resolution.}
\label{fig:bracketing-endpoints}
\end{figure}

Over \(10{,}000\) time units,
\[
\max\|p(t)\|\leq5.81\quad(\eta=0),
\qquad
\max\|p(t)\|\leq1.77\quad(\eta=1),
\]
where the latter maximum occurs during the initial approach; subsequent \(1000\)-time-unit windows have maximum norm approximately \(1.131\).  The Poincar\'e section \(x_1=0\) with upward crossings gives an extended set at \(\eta=0\) and one numerical point per period at \(\eta=1\).  The normalized spectral entropy of \(x_2(t)\) is \(0.6143\) and \(0.1854\), respectively.  Approximately \(3080\) Fourier bins are required to collect \(90\%\) of the power at \(\eta=0\), compared with \(7\) at \(\eta=1\).

\begin{figure}[!t]
\centering
\includegraphics[width=.84\linewidth]{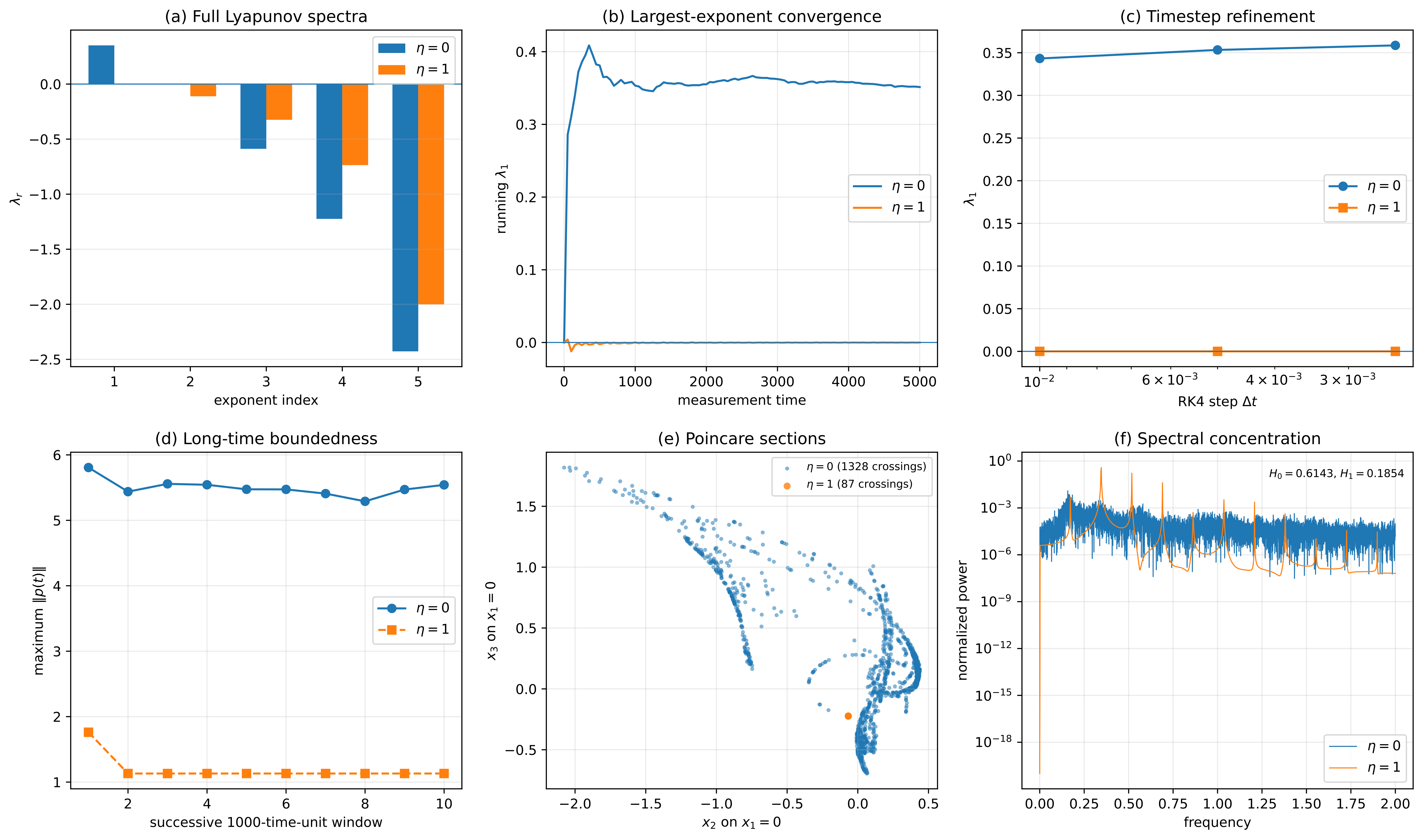}
\caption{Numerical diagnostics for the two endpoint flows: Lyapunov spectra and convergence, timestep refinement, long-time norm bounds, Poincar\'e sections and spectral concentration.}
\label{fig:bracketing-validation}
\end{figure}

Together, the positive largest exponent, near-zero flow exponent, mean-divergence agreement, extended Poincar\'e section and broadband spectrum provide mutually consistent numerical evidence for a bounded chaotic invariant set at \(\eta=0\).  The right-bracketed endpoint is periodic.  Establishing a hyperbolic invariant set at \(\eta=0\) would require a separate computer-assisted argument.

\subsubsection{Periodic orbit and Floquet stability at the right-bracketed endpoint}

A shooting calculation with a phase condition and simultaneous variational integration gives the point
\begin{equation}
\begin{aligned}
p_{\mathrm{per}}={}&0.352211456319145
+0.280943768469983i
+0.580647718066596j\\
&+0.274403530427830k
+0.442578560477541\ell
\end{aligned}
\label{eq:bracketing-periodic-point}
\end{equation}
on a periodic orbit of period
\begin{equation}
T=5.795118818486838.
\label{eq:bracketing-period}
\end{equation}
Newton correction using the monodromy matrix reduces the shooting residual to \(5.7\times10^{-15}\).  The Floquet multipliers are
\begin{equation}
1,
\quad0.526353763,
\quad-0.152037463,
\quad0.013976494,
\quad-9.21838\times10^{-6}.
\label{eq:bracketing-floquet}
\end{equation}
At numerical precision, one multiplier is neutral in the phase direction and all four transverse multipliers lie strictly inside the unit circle.  The computed orbit is therefore orbitally asymptotically stable.

A multistart equilibrium search at \(\eta=0\) locates four real equilibria.  The irregular trajectory repeatedly passes through neighborhoods of the origin and two additional saddles; nearest-equilibrium assignment gives visitation fractions approximately \(0.630,0.208,0,0.162\).  These statistics are consistent with a saddle-mediated stretching-and-reinjection mechanism.  A rigorous homoclinic or horseshoe description would require additional global analysis.

\subsubsection{Associator balance and the saddle-node of periodic orbits}

For \(N(p)=\|p\|^2\), the norm rate decomposes exactly as
\begin{equation}
\frac{dN}{dt}
=2\langle p,Ap\rangle
+2\langle p,B(pp)\rangle
+2\eta\langle p,[B,p,p]\rangle.
\label{eq:bracketing-radial-balance}
\end{equation}
The endpoint averages are
\begin{table}[H]
\centering
\small
\begin{tabular}{crrrr}
\toprule
$\eta$ & linear & left quadratic & associator & total\\
\midrule
0 & -0.199331 & 0.199407 & 0 & $7.62\times10^{-5}$\\
1 & -0.099006 & -0.247837 & 0.346843 & $1.79\times10^{-14}$\\
\bottomrule
\end{tabular}
\caption{Mean contributions to \(dN/dt\).  On the right-bracketed cycle, the associator supplies the positive nonlinear contribution that balances the two negative terms.}
\label{tab:bracketing-radial}
\end{table}
At \(\eta=0\), linear loss and intermittent left-quadratic input balance on average.  At \(\eta=1\), the linear and left-bracketed contributions are both negative, and the associator closes the cycle balance.  Along the cycle, the mean norm of the associator term is approximately \(0.8465\), and its mean magnitude is about \(61\%\) of the total vector-field magnitude.  The periodic orbit is therefore sustained by an order-one nonassociative contribution rather than by a weak correction to the left-bracketed flow.

Pseudo-arclength shooting continuation follows the stable periodic orbit from \(\eta=1\) toward smaller values, turns at
\begin{equation}
\eta_{\mathrm{SN}}\approx0.50903065,
\qquad
T_{\mathrm{SN}}\approx4.947,
\label{eq:bracketing-saddle-node}
\end{equation}
and reveals a companion unstable branch.  A nontrivial real Floquet multiplier passes through \(+1\) at the turning point, identifying a saddle-node bifurcation of periodic orbits.  Hence the attracting recurrent state at \(\eta=1\) is created by a nonlinear periodic-orbit bifurcation, not by stabilization of the common saddle-focus at the origin.

\begin{figure}[!b]
\centering
\includegraphics[width=.99\linewidth]{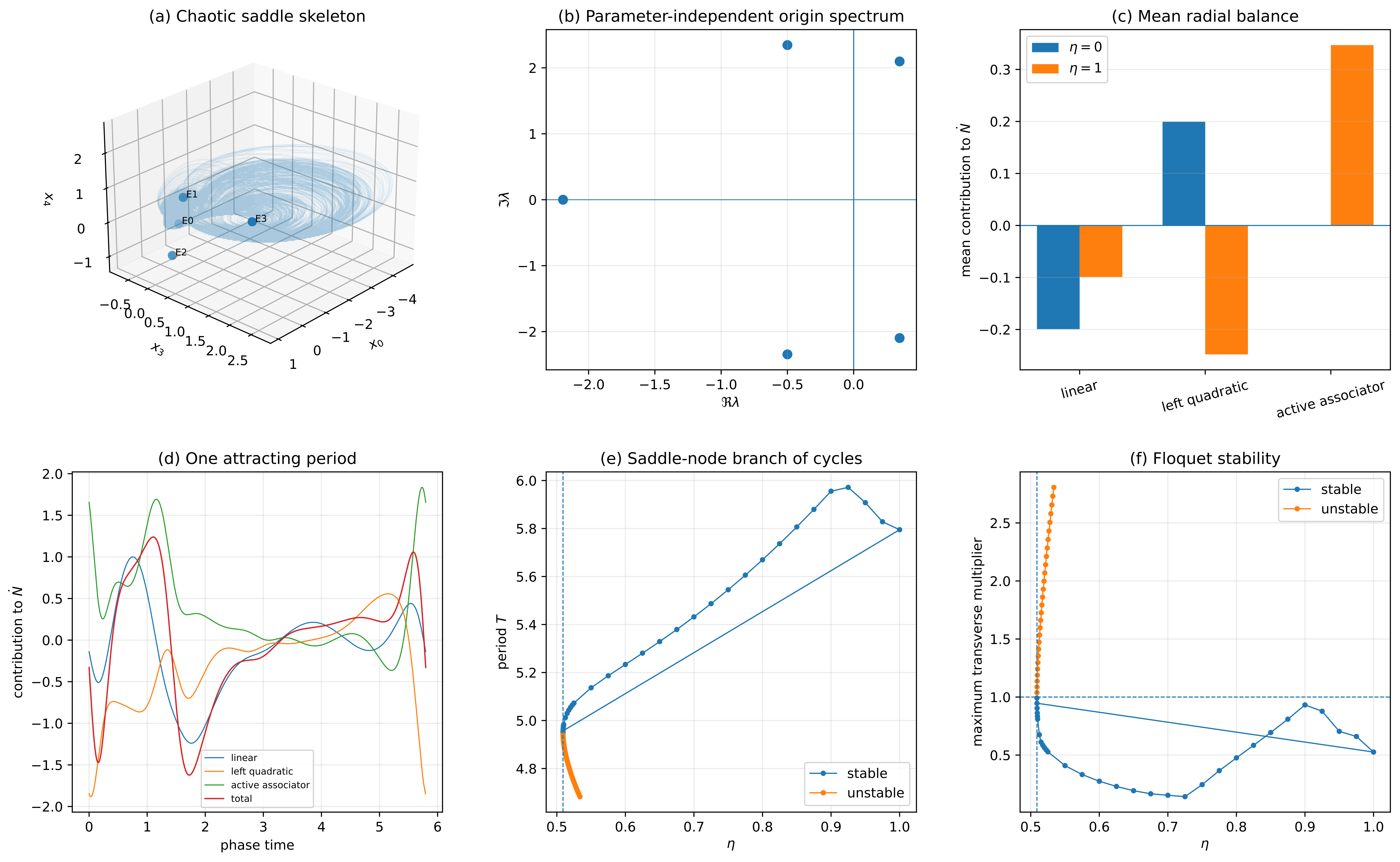}
\caption{Local and global organization of the bracketing family: the parameter-independent saddle-focus, saddle visitation, the associator-resolved radial balance, and the stable and unstable periodic branches meeting at a fold.}
\label{fig:bracketing-mechanism}
\end{figure}

The exact algebraic identities show that the parameter changes only the nonlinear associator contribution, while the numerical continuation identifies the corresponding global reorganization.  The common saddle-focus persists throughout the family, but the radial energy budget changes and a stable periodic branch is created at a fold.  At the two endpoint bracketings, the computations support bounded chaotic dynamics on the left and periodic recurrence on the right.

\subsection{Consequences for Pentarnion dynamics}
\label{subsec:ode-discussion}

Together with the rank-four theory of Section~\ref{sec:rank-four-riccati}, the elliptic equation and the bracketing family reveal three different ways in which the fixed Pentarnion multiplication enters differential equations.  First, bracket-complete words provide a compact encoding of structured five-dimensional quadratic fields and the flexibility defect can raise the one-generated algebraic rank, activating an exactly integrable three-dimensional flow.  Second, in Equation~\eqref{eq:compact-genus-one-pentarnion}, two ordered products constrain an entire Jacobi family and select the modulus \(m=1/2\); the resulting ambient flow is irreducible under every fixed linear decomposition and lies outside the projective and conformal Riccati classes.  Third, in Equation~\eqref{eq:bracketing-compact}, the same multiplication makes two parenthesizations inequivalent, and the resulting associator changes the global recurrent dynamics without altering the linearisation at the origin.

These mechanisms are complementary rather than interchangeable.  The genus-one flow supplies exact full-dimensional integrability, while the bracketing family isolates a directly measurable nonassociative contribution to the creation and maintenance of a periodic orbit.  The appendix strengthens this distinction by giving a reducible solvable flow whose exact recurrence exists only at the associator-compensated value: full affine span alone does not imply irreducibility, but the periodic orbit remains sensitive to nonassociativity.

\section*{Acknowledgements}
AP acknowledges the University of Hertfordshire for providing the research environment in which this work was completed. The ideas developed in this paper originated in investigations of algebra-valued formulations of polynomial differential equations motivated by the generalisation of the Hodgkin–Huxley model. AP was also inspired by Sir William Rowan Hamilton’s \textit{Lectures on Quaternions} and by the mathematical writings of Sir Isaac Newton.

The authors would like to thank S. Walcher for some fruitful discussions on the content of the paper.

\appendix
\raggedbottom

\section{Associator compensation and destruction of a \texorpdfstring{$2{:}3{:}5$}{2:3:5} orbit}
\label{app:associator-235}

This appendix isolates, in an exactly solvable setting, the effect of leaving an associator uncompensated.  At one parameter value the equation separates into two Riccati-type subsystems and possesses a full-affine-span periodic orbit with frequencies \(2{:}3{:}5\).  The orbit is not robust under the nonassociative deformation: for every other parameter value, the associator changes the initial tangent and destroys the displayed periodic solution.

For \(p=w+xi+yj+zk+u\ell\), define
\begin{equation}
\label{eq:app-qr-splitting}
 q=-\frac12 i(ip+pi)=w+xi,
 \qquad
 r=p-q=yj+zk+u\ell.
\end{equation}
The splitting is recovered through multiplication by the canonical unit \(i\); no alternative basis is introduced.  For \(0\leq\eta\leq1\), consider
\begin{equation}
\label{eq:app-associator-family}
\begin{aligned}
\dot p={}&2i\bigl(q(1-q)\bigr)\\
&+\frac32\left[j(kr+rk)-(jr+rj)k\right]\\
&+\frac52\left[\ell-(r\ell)r-\eta[\ell,r,r]\right].
\end{aligned}
\end{equation}
Writing its vector field as \(F_\eta\),
\begin{equation}
\label{eq:app-associator-difference}
F_\eta(p)=F_1(p)+\frac52(1-\eta)[\ell,r,r].
\end{equation}
Thus \(\eta=1\) is the associator-compensated equation, whereas \(\eta<1\) leaves a nonzero fraction of the bracketing defect in the dynamics.

\begin{proposition}[Compensated block decomposition]
\label{prop:app-compensated-components}
At \(\eta=1\), Equation~\eqref{eq:app-associator-family} is equivalent to
\begin{equation}
\label{eq:app-compensated-components}
\begin{aligned}
\dot w&=-2x+4wx,\\
\dot x&=2w-2w^2+2x^2,\\
\dot y&=-3z+5uy,\\
\dot z&=3y+5uz,\\
\dot u&=\frac52\left(1-y^2-z^2+u^2\right).
\end{aligned}
\end{equation}
Equivalently, we have
\begin{equation}
\label{eq:app-product-blocks}
 \dot q=2iq(1-q),
 \qquad
 \dot r=\frac32\left[j(kr+rk)-(jr+rj)k\right]
 +\frac52\left[\ell-(r\ell)r-[\ell,r,r]\right].
\end{equation}
The compensated flow is therefore the direct product of a two-dimensional complex Riccati equation and a three-dimensional conformal Riccati equation.  The latter belongs to the class studied in \cite{deLucasTobolskiVilarino2016}.
\end{proposition}

\begin{proof}
Since \(q\in\operatorname{span}\{1,i\}\), expansion of \(2iq(1-q)\) gives the first two equations.  For \(r=yj+zk+u\ell\), anticommutativity gives
\[
j(kr+rk)-(jr+rj)k=-2zj+2yk.
\]
Moreover, \(r^2=-(y^2+z^2+u^2)\) and \(r\ell+\ell r=-2u\).  Hence
\[
\ell-(r\ell)r-[\ell,r,r]
=(1-y^2-z^2-u^2)\ell+2ur,
\]
which gives the last three equations and proves that the \(q\)- and \(r\)-subsystems close independently.
\end{proof}

\begin{theorem}[Exact full-span periodic solution]
\label{thm:app-associator-solution}
For
\begin{equation}
\label{eq:app-associator-p0}
 p(0)=p_0=\frac13+\frac1{\sqrt3}j,
\end{equation}
the compensated equation \(F_1\) has the solution
\begin{equation}
\label{eq:app-associator-orbit}
\begin{aligned}
p_*(t)={}&
\frac{1+2\cos2t}{5+4\cos2t}
+\frac{2\sin2t}{5+4\cos2t}\,i\\
&+\frac{\sqrt3\cos3t}{2+\cos5t}\,j
+\frac{\sqrt3\sin3t}{2+\cos5t}\,k
+\frac{\sin5t}{2+\cos5t}\,\ell.
\end{aligned}
\end{equation}
The solution is globally regular and embedded, has minimal period \(2\pi\), and its affine span is all of \(\Pent\).
\end{theorem}

\begin{proof}
For the complex block, put \(h=q/(1-q)\).  Equation~\eqref{eq:app-product-blocks} gives \(\dot h=2ih\).  Since \(q(0)=1/3\), one has \(h(0)=1/2\), and therefore
\[
q(t)=\frac{e^{2it}}{2+e^{2it}}.
\]
Rationalizing the denominator gives the first two components of equation~\eqref{eq:app-associator-orbit}.

For the three-dimensional block, let \(n=y^2+z^2+u^2\) and define
\[
X_0=\frac{1-n}{1+n},\qquad
X_1=\frac{2y}{1+n},\qquad
X_2=\frac{2z}{1+n},\qquad
X_3=\frac{2u}{1+n}.
\]
Direct differentiation gives
\[
\dot X_1=-3X_2,\qquad \dot X_2=3X_1,
\qquad
\dot X_0=-5X_3,\qquad \dot X_3=5X_0.
\]
The initial values are \((X_0,X_1,X_2,X_3)=(1/2,\sqrt3/2,0,0)\).  Inverting the stereographic map gives the final three components in Equation~\eqref{eq:app-associator-orbit}.  Both denominators are bounded below by one.

The first block determines \(e^{2it}\), whereas the phase of \(y+iz\) determines \(e^{3it}\).  Equality \(p_*(t_1)=p_*(t_2)\) therefore implies both
\(e^{2i(t_1-t_2)}=1\) and \(e^{3i(t_1-t_2)}=1\).  Since \(\gcd(2,3)=1\), the parametrization is embedded and has minimal period \(2\pi\).

For affine span, set \(\zeta=e^{it}\).  The combinations \(w+ix\) and \(w-ix\) have poles at \(\zeta^2=-2\) and \(\zeta^2=-1/2\), respectively.  These poles are disjoint from those of the last three coordinates, whose common denominator is
\[
P(\zeta)=\zeta^{10}+4\zeta^5+1.
\]
Any affine identity first forces the coefficients of \(w\) and \(x\) to vanish.  After multiplication by \(P\), the remaining numerator is a linear combination of the distinct powers \(\zeta^{10},\zeta^8,\zeta^5,\zeta^2,1\), so all remaining coefficients vanish.
\end{proof}

\begin{proposition}[The uncompensated associator destroys the exact orbit]
\label{prop:app-associator-destroys-orbit}
For every \(0\leq\eta<1\), the curve \(p_*(t)\) in Equation~\eqref{eq:app-associator-orbit} is not a solution of Equation~\eqref{eq:app-associator-family} with the same initial value \(p_0\).  More precisely,
\begin{equation}
\label{eq:app-initial-tangent-defect}
F_\eta(p_0)-\dot p_*(0)
=\frac{5}{12}(1-\eta)(-i+k+2\ell)\neq0.
\end{equation}
\end{proposition}

\begin{proof}
At \(p_0\),
\[
r_0=\frac1{\sqrt3}j.
\]
A direct multiplication gives
\begin{equation}
\label{eq:app-initial-associator}
[\ell,r_0,r_0]=\frac16(-i+k+2\ell).
\end{equation}
Since \(\dot p_*(0)=F_1(p_0)\), substitution into Equation~\eqref{eq:app-associator-difference} yields Equation~\eqref{eq:app-initial-tangent-defect}.  Uniqueness of solutions for polynomial vector fields then excludes \(p_*\) from every member of the family with \(\eta<1\).
\end{proof}

For a general \(r=yj+zk+u\ell\), the associator has the canonical expansion
\begin{equation}
\label{eq:app-general-associator}
\begin{aligned}
[\ell,r,r]={}&-\sqrt2\,yz
-\frac12(y+z)(u+y-z)i\\
&-\frac12(4uy+uz+yz+z^2)j\\
&+\frac12(-uy-4uz+y^2+yz)k
+(y^2+z^2)\ell.
\end{aligned}
\end{equation}
The scalar and $i$-components in Equation~\eqref{eq:app-general-associator} feed the $r$-variables back into the $q$-sector as soon as $\eta<1$.  Thus the same term that changes the tangent in Equation~\eqref{eq:app-initial-tangent-defect} also destroys the direct-product decomposition of the compensated equation.  For example, integration from $p_0$ at $\eta=0.9$ gives
\[
 \|p_{0.9}(2\pi)-p_0\|\approx0.607545,
\]
whereas the compensated solution returns exactly after $2\pi$.  This computation illustrates the departure established algebraically by Proposition~\ref{prop:app-associator-destroys-orbit}; it is not used in the proof.

\begin{figure}[!t]
\centering
\includegraphics[width=.99\linewidth]{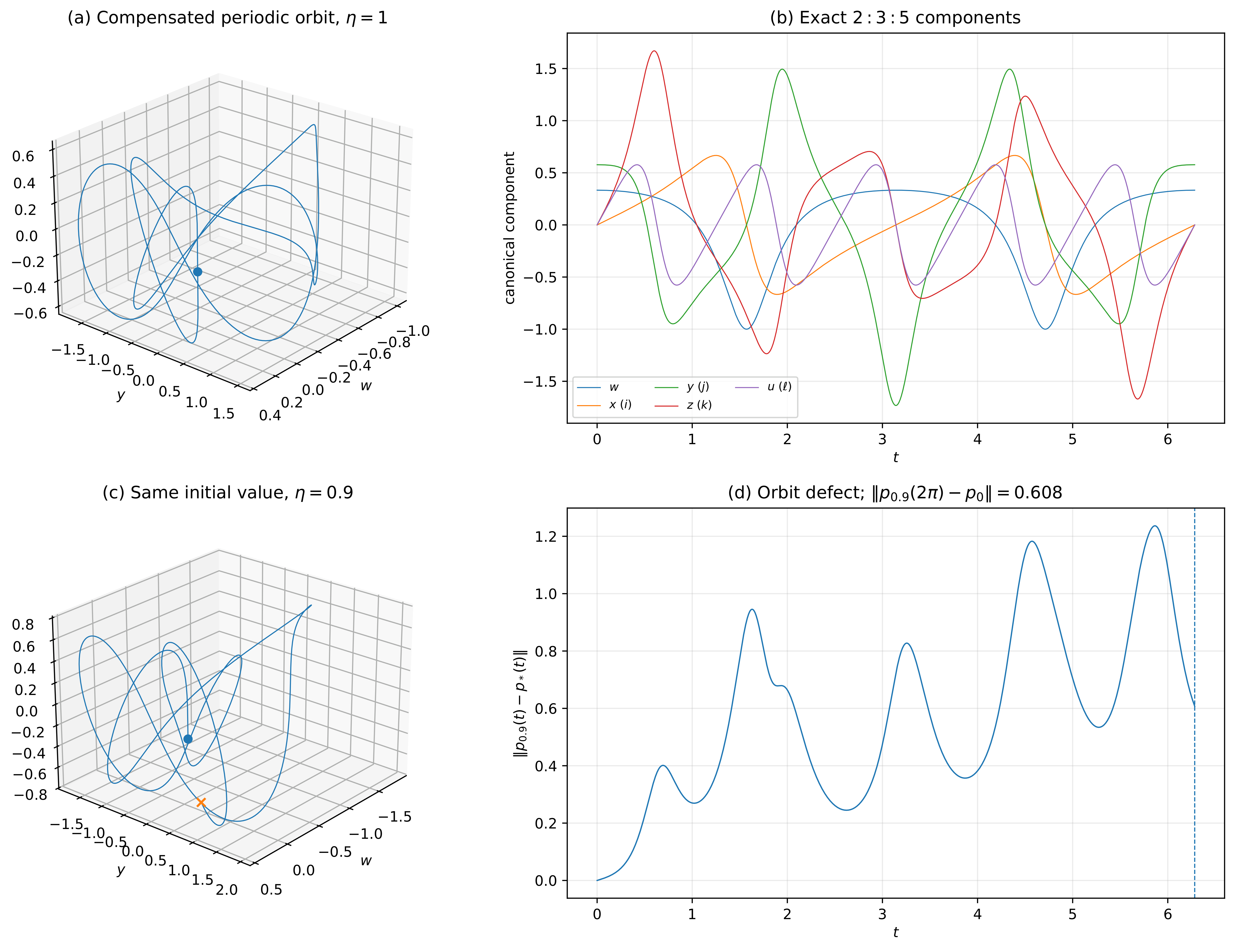}
\caption{Associator compensation and loss of exact recurrence.  Top left: the exact \(\eta=1\) orbit in the canonical \((w,y,u)\) projection.  Top right: its five components over one period.  Bottom left: the trajectory from the same initial value at \(\eta=0.9\), where the uncompensated associator couples the two blocks.  Bottom right: the deviation from the exact orbit; after one nominal period, the trajectory does not return to \(p_0\).}
\label{fig:app-associator-235}
\end{figure}

The compensated member is reducible and is therefore not a second instance of the irreducible elliptic mechanism developed in Section~\ref{subsec:genus-one-pentarnion}.  Its role is different: it gives an exact, transparent example in which nonassociativity is dynamically observable.  Compensation produces a closed full-span orbit, whereas any residual associator immediately changes the vector field along that orbit and prevents its persistence.

\section{Code and data availability}
\label{app:code-hierarchy}

The code and supporting data for this article are available in
\path{pentarnion-dynamics/} of the repository
\url{https://github.com/poshtkohi/pentarnion-algebra}.
The guide in that folder describes the script dependencies, software
requirements and instructions for running the symbolic verifications
and numerical calculations.

\flushbottom
\providecommand{\bysame}{\leavevmode\hbox to3em{\hrulefill}\thinspace}
\providecommand{\MR}{\relax\ifhmode\unskip\space\fi MR }
\providecommand{\MRhref}[2]{\href{http://www.ams.org/mathscinet-getitem?mr=#1}{#2}
}
\providecommand{\href}[2]{#2}

\end{document}